\documentclass[12pt,a4paper]{article}
\usepackage[cp1251]{inputenc}
\usepackage[OT1]{fontenc}
\usepackage{amsmath}
\usepackage{amsfonts}
\usepackage{amssymb}
\usepackage{amsthm}
\usepackage[left=2cm,right=2cm,top=2cm,bottom=2cm]{geometry}
\usepackage{enumitem,multicol}
\usepackage{float}

\numberwithin{equation}{section}

\newtheorem{theorem}{Theorem}
\newtheorem*{remark}{Remark}
\newtheorem{lemma}{Lemma}
\newtheorem{corollary}{Corollary}
\newtheorem{definition}{Definition}
\newtheorem{prop}{Proposition}
\newtheorem{exercise}{Exercise}

\newcounter{cntr}

\def\const{\operatorname{const}}
\def\spann{\operatorname{span}}
\def\Span{\operatorname{span}}
\def\sgn{\operatorname{sgn}}
\def\Vec{\operatorname{Vec}}
\def\intt{\operatorname{int}}
\def\rank{\operatorname{rank}}
\def\Ker{\operatorname{Ker}}
\def\Im{\operatorname{Im}}
\def\Lie{\operatorname{Lie}}
\def\SO{\operatorname{SO}}
\def\SE{\operatorname{SE}}
\def\Id{\operatorname{Id}}
\def\cl{\operatorname{cl}}
\def\Lip{\operatorname{Lip}}
\def\Exp{\operatorname{Exp}}
\def\Max{\operatorname{Max}}
\def\conj{\operatorname{conj}}

\def\T{{\mathbb T}}
\def\R{{\mathbb R}}
\def\N{{\mathbb N}}
\title{Introduction to geometric control}
\author{Yuri Sachkov\\
Program Systems Institute\\
Pereslavl-Zalessky\\
Russia\\
{\tt yusachkov@gmail.com}}

\begin{document}

\maketitle

\begin{abstract}
Lecture notes of a short course on geometric control theory given in Brasov, Romania (August 2018) and in Jyv\"askyl\"a, Finland (February 2019).
\end{abstract}

\tableofcontents
%\pagebreak
%%%%%%%%%%%%%%%%%%% стр. 1 %%%%%%%%%%%%%%%%%%%
\section{Introduction}
\subsection{Examples of optimal control problems}
We state several optimal control problems, many of which we study in the sequel.
\paragraph{Example 1: Stopping a train}
Consider a material point of mass $m>0$ with coordinate $x_1 \in \mathbb{R}$ that moves under the action of a force $F$ bounded by absolute value by $F_{max}>0$. Given an initial position $x_0$ and initial velocity $\dot{x}_0$ of the material point, we should find a force $F$ that steers the point to the origin with zero velocity, for a minimal time.

The second law of Newton gives $\lvert m \ddot{x}_1 \rvert = \lvert F \rvert \le F_{max}$, thus $\lvert \ddot{x}_1 \rvert \le \frac{F_{max}}{m}$. 
%%%%%%%%%%%%%%%%%%% стр. 2 %%%%%%%%%%%%%%%%%%%
Choosing appropriate units of measure, we can obtain $\frac{F_{max}}{m} = 1$, thus  $\lvert \ddot{x}_1 \rvert \le 1$. Denote velocity of the point $x_2 = \dot{x}_1$, and acceleration $\dot{x}_2 = u$, $\lvert u \rvert \le 1$. Then the problem is formalized as follows:
\begin{align*}
&\dot{x}_1 = x_2, \quad x = (x_1, x_2) \in \mathbb{R}^2, \\
&\dot{x}_2 = u, \quad \lvert u \rvert \le 1, \\
&x(0) = (x_0, \dot{x}_0), \quad x(t_1) = (0, 0),\\
&t_1 \to \min.
\end{align*}
This is an example of a linear time-optimal problem.

\paragraph{Example 2: Control of linear oscillator}
Consider a pendulum that performs small oscilla\-tions under the action of a force bounded by absolute value. We should choose a force that steers the pendulum from an arbitrary position and velocity to the stable equilibrium for a minimum time. After choosing appropriate units of 
%%%%%%%%%%%%%%%%%%% стр. 3 %%%%%%%%%%%%%%%%%%%
measure, we get a mathematical model: $\ddot{x}_1 = -x_1 + u$, $\lvert u \rvert \le 1$, $x_1 \in \mathbb{R}$. Introducing the notation $x_2 = \dot{x}_1$, we get a linear time-optimal problem:
\begin{align*}
&\dot{x}_1 = x_2, \quad x = (x_1, x_2) \in \mathbb{R}^2, \\
&\dot{x}_2 = -x_1 + u, \quad \lvert u \rvert \le 1, \\
&x(0) = x^0, \quad x(t_1) = 0, \\
&t_1 \to \min.
\end{align*}

\paragraph{Example 3: Markov-Dubins car}
Consider a simplified model of a car that is given by a unit vector attached at a point $(x, y) \in \mathbb{R}^2$, with orientation $\theta \in S^1$. The car moves forward with unit velocity and can simultaneously rotate with angular velocity $\lvert \dot{\theta} \rvert \le 1$. Given an initial and a terminal state of the car, we should choose the angular velocity in such a way that the time of motion is minimum possible.
%%%%%%%%%%%%%%%%%%% стр. 4 %%%%%%%%%%%%%%%%%%%

We have the following nonlinear time-optimal problem:
\begin{align*}
&\dot{x} = \cos \theta, \quad q = (x, y, \theta) \in \mathbb{R}^2_{x,y} \times S^1_{\theta} = M, \\
&\dot{y} = \sin \theta, \quad \lvert u \rvert \le 1, \\
&\dot{\theta} = u, \\
&q(0) = q_0, \quad q(t_1) = q_1, \\ 
&t_1 \to \min.
\end{align*}
Notice that in this problem the state space $M = \mathbb{R}^2 \times S^1$ is a non-trivial smooth manifold, homeomorphic to the solid torus.

\paragraph{Example 4: Reeds-Shepp car}
Consider a model of a (more realistic) car in the plane that can move forward or backward with arbitrary linear velocity and simultaneously rotate with arbitrary angular velocity.
%%%%%%%%%%%%%%%%%%% стр. 5 %%%%%%%%%%%%%%%%%%%
The state of the car is given by its position in the plane and orientation angle. We should find a motion of the car from a given initial state to a given terminal state, so that the length of the path in the space of positions and orientations was minimum possible.

We get the following optimal control problem:
\begin{align*}
&\dot{x} = u \cos \theta, \quad q = (x, y, \theta) \in \mathbb{R}^2_{x,y} \times S^1_{\theta}, \\
&\dot{y} = u \sin \theta, \quad (u, v) \in \mathbb{R}^2, \\
&\dot{\theta} = v, \\
&q(0) = q_0, \quad q(t_1) = q_1,\\
&l = \int^{t_1}_0 \sqrt{\dot{x}^2 + \dot{y}^2 + \dot{\theta}^2} dt = \int^{t_1}_0 \sqrt{u^2 + v^2} dt \to \min.
\end{align*}
This is an example of an optimal control problem with integral cost functional.
%%%%%%%%%%%%%%%%%%% стр. 6 %%%%%%%%%%%%%%%%%%%
\paragraph{Example 5: Euler elasticae}
Consider a uniform elastic rod of length $l$ in the plane. Suppose that the rod has fixed endpoints and tangents at endpoints. We should find the profile of the rod.

Let $(x(t), y(t))$ be an arclength parameterization of the rod, and let $\theta (t)$ be its orientation angle in the plane. Then the rod satisfies the following conditions:
\begin{align*}
&\dot{x} = \cos \theta, \quad q = (x, y, \theta) \in \mathbb{R}^2 \times S^1, \\
&\dot{y} = \sin \theta, \quad u \in \mathbb{R}, \\
&\dot{\theta} = u, \\
&q(0) = q_0, \quad q(t_1) = q_1, \quad t_1 = l \textrm{ is the length of the rod}.
\end{align*}
Elastic energy of the rod is $J = \frac12 \int^{t_1}_0 k^2 dt$, while $k$ is the curvature of the rod. Since for an arclength parameterized rod $k = \dot{\theta} = u$, we obtain a cost functional 
%%%%%%%%%%%%%%%%%%% стр. 7 %%%%%%%%%%%%%%%%%%%
$$J = \frac12 \int^{t_1}_0 u^2 dt \to \min,
$$ 
since the rod takes the form that minimizes its elastic energy.
\paragraph{Example 6: Sphere rolling on a plane without slipping or twisting}
Let a uniform sphere roll without slipping or twisting on a horizontal plane. One can imagine that the sphere rolls between two horizontal planes: fixed lower one and moving upper one. The state of the system is determined by the contact point of the sphere and the plane, and orientation of the sphere in the space. We should roll the sphere from a given initial state to a given terminal state, so that the length of the curve in the plane traced by the
%%%%%%%%%%%%%%%%%%% стр. 8 %%%%%%%%%%%%%%%%%%%
contact point was the shortest possible.

Let $(x,y)$ denote coordinates of the contact point of the sphere with the plane. Introduce a fixed orthonormal frame $(e_1, e_2, e_3)$ in the space such that $e_1$ and $e_2$ are contained in plane, and a moving orthonormal frame $(f_1, f_2, f_3)$ attached to the sphere. Let a point of the sphere have coordinates $(x,y,z)$ in the  fixed frame $(e_1, e_2, e_3)$, and coordinates $(X,Y,Z)$ in the moving frame $(f_1, f_2, f_3)$, i.e., 
$$xe_1 + ye_2 + ze_3 = Xf_1 + Yf_2 + Zf_3.$$
Then the orthogonal matrix $R$ such that
\begin{align*}
R
\begin{pmatrix}
x\\
y\\
z
\end{pmatrix}
=
\begin{pmatrix}
X\\
Y\\
Z
\end{pmatrix}
\end{align*}
determines orientation of the sphere in the space. We have 
$$ R \in \SO(3) = \{ A \in \mathbb{R}^{3 \times 3} \mid A^T = A^{-1},\quad \det A = 1 \}. $$
%%%%%%%%%%%%%%%%%%% стр. 9 %%%%%%%%%%%%%%%%%%%
Then our problem is written as follows:
\begin{align*}
&\dot{x} = u, \qquad q = (x,y,R) \in \mathbb{R}^2 \times \SO(3),\\
&\dot{y} = v, \qquad (u, v) \in \mathbb{R}^2,\\
&\dot{R} = R 
\begin{pmatrix}
0 & 0 & -u\\
0 & 0 & -v\\
u & v & 0
\end{pmatrix},\\
&q(0) = q_0, \quad q(t_1) = q_1, \\
&l = \int^{t_1}_0 \sqrt{u^2 + v^2} dt \to \min.
\end{align*}
\paragraph{Example 7: Antropomorphic curve reconstruction}
Suppose that a greyscale image is given by a set of isophotes (level lines of brightness). Let the image be corrupted in some domain, and our goal is to reconstruct it antropomorphically, i.e., close to the way a human brain does. Consider a particular problem of antropomorphic reconstruction
%%%%%%%%%%%%%%%%%%% стр. 10 %%%%%%%%%%%%%%%%%%%
of a curve.

According to a discovery of Hubel and Wiesel (Nobel prize 1981), a human brain stores curves not as sequences of planar points $(x_i, y_i)$, but as sequences of positions and orientations $(x_i, y_i, \theta_i)$. Moreover, an established model of the primary visual cortex $V1$ of the human brain states that corrupted curves of images are reconstructed according to a variational principle, i.e., in a way that minimizes the activation energy of neurons required for drawing the missing part of the curve.

So the discovery by Hubel and Wiesel states that the human brain lifts images $(x(t), y(t))$ from the plane to the space of positions and orientations $(x(t), y(t), \theta(t))$. The lifted
%%%%%%%%%%%%%%%%%%% стр. 11 %%%%%%%%%%%%%%%%%%%
curve is a solution to the control system
\begin{align*}
&\dot{x} = u \cos \theta, \quad q = (x, y, \theta) \in \mathbb{R}^2 \times S^1, \\
&\dot{y} = u \sin \theta,\\
&\dot{\theta} = v,
\end{align*}
with the boundary conditions provided by endpoints and tangents of the corrupted curve:
\begin{align*}
 q(0) = q_0, \quad q(t_1) = q_1. 
\end{align*}
Moreover, the activation energy of neurons required to draw the corrupted curve is given by the integral to be minimized:
\begin{align*}
J = \int^{t_1}_0 (u^2 + v^2)dt \to \min.
\end{align*}
By Cauchy-Schwarz inequality, minimization of the energy $J$ is equivalent to minimization of the length functional
\begin{align*}
l = \int^{t_1}_0 \sqrt{u^2 + v^2} dt \to \min.
\end{align*}
%%%%%%%%%%%%%%%%%%% стр. 12 %%%%%%%%%%%%%%%%%%%
We have a remarkable fact: optimal trajectories for the Reeds-Shepp car provide solutions to the problem of antropomorphic curve reconstruction.
\subsection{Control systems and problems}
\subsubsection{Dynamical systems and control systems}
A smooth dynamical system, or an ODE on a smooth manifold, is given by an equation
\begin{align}
\label{DS}
\dot{q} = f(q), \quad q \in M,
\end{align}
where $f \in \Vec(M)$ is a smooth vector field on $M$. A basic property of a dynamical system is that it is deterministic, i.e., given an initial condition $q(0) = q_0$ and a time $t > 0$, there exists a unique solution $q(t)$ to ODE \eqref{DS}.

A control system is obtained from dynamical system \eqref{DS} if we add a
%%%%%%%%%%%%%%%%%%% стр. 13 %%%%%%%%%%%%%%%%%%%
control parameter $u$ in the right-hand side:
\begin{align}
\label{CS}
\dot{q} = f(q, u), \quad q \in M, \quad u \in U.
\end{align}
The control parameter varies in a set of control parameters $U$ (usually a subset of $\mathbb{R}^m$). This parameter can change in time: we can choose a function $u = u(t) \in U$ and substitute it to the right-hand side of control system \eqref{CS} to obtain a nonautonomous ODE
\begin{align}
\label{NODE}
\dot{q} = f(q, u(t)).
\end{align}
Together with an initial condition
\begin{align}
\label{init}
q(0) = q_0,
\end{align}
ODE \eqref{NODE} determines a unique solution --- a trajectory $q_u (t)$, $t> 0$, of control system \eqref{CS} corresponding to the control $u(t)$ and initial condition \eqref{init}.
%%%%%%%%%%%%%%%%%%% стр. 14 %%%%%%%%%%%%%%%%%%%

For another control $\tilde{u}(t)$, we get another trajectory $q_{\tilde{u}}(t)$ with initial condition \eqref{CS}.

Regularity assumptions for control $u(\cdot)$ can vary from a problem to a problem; typical examples are piecewise constant controls or Lebesgue measurable bounded controls. The controls considered in a particular problem are called admissible controls.

If we fix initial condition \eqref{init} and vary admissible controls, we get a new object --- attainable set of control system \eqref{CS} for arbitrary times:
\begin{align*}
A_{q_0} = \{ q_u (t) \mid q_u(0) = q_0, \quad u \in L^{\infty} ([0, +\infty), U) \}.
\end{align*}
For a dynamical system, the attainable set is not considered since it is just a positive-time half-trajectory. But for control systems, the attainable set is a non-trivial object, and its study is one of the central problems
%%%%%%%%%%%%%%%%%%% стр. 15 %%%%%%%%%%%%%%%%%%%
of control theory.

If we apply restrictions on the terminal time of trajectories, we get restricted attainable sets:
\begin{align*}
&A_{q_0} (T) = \{ q_u (T)  \mid  q_u(0) = q_0, \quad u \in L^{\infty} ([0, T], U) \},\\
&A_{q_0} (\le T)  = \bigcup^T_{t=0} A_{q_0} (t).
\end{align*}

\subsubsection{Controllability problem}
\begin{definition}
A control system \eqref{CS} is called:
\begin{itemize}
\item globally (completely) controllable, if $A_{q_0} = M$ for any $q_0 \in M$,
\item globally controllable from a point $q_0 \in M$ if $A_{q_0} = M$,
\item locally controllable at $q_0$ if $q_0 \in \intt A_{q_0}$,
\item small time locally controllable (STLC) at $q_0$ if $q_0 \in \intt A_{q_0} (\le T)$ for any $T > 0$.
\end{itemize}
\end{definition}
%%%%%%%%%%%%%%%%%%% стр. 16 %%%%%%%%%%%%%%%%%%%
Even the local controllability problem is rather hard to solve: there exist necessary conditions and sufficient conditions for STLC for arbitrary dimension of the state space $M$, but local controllability tests  are available only for the case $\dim M = 2$. The global controllability problem is naturally much more harder: there exist global controllability conditions only for very symmetric systems: linear systems, left-invariant systems on Lie groups.

\subsubsection{Optimal control problem}
Suppose that for control system \eqref{CS} the controllability problem between points $q_0, q_1 \in M$ is solved positively. Then typically the points $q_0, q_1$ are connected by more that one trajectory 
%%%%%%%%%%%%%%%%%%% стр. 17 %%%%%%%%%%%%%%%%%%%
of the control system (usually by continuum of trajectories). Then there naturally arises the question of the best (optimal in a certain sense) trajectory connecting $q_0$ and $q_1$. In order to measure the quality of trajectories (controls), introduce a cost functional to be minimized: $J = \int^{t_1}_0 \varphi (q, u) dt$.  Thus we get an optimal control problem:
\begin{align*}
&\dot{q} = f(q, u), \quad q \in M, \quad u \in U,\\
&q(0) = q_0, \quad q(t_1) = q_1,\\
&J = \int^{t_1}_0 \varphi (q, u) dt \to \min.
\end{align*}
Here the terminal time $t_1$ may be fixed or free.

The optimal control problem is also rather hard to solve --- this is an optimization problem in an infinite-dimensional space. There exist general necessary optimality
%%%%%%%%%%%%%%%%%%% стр. 18 %%%%%%%%%%%%%%%%%%%
conditions (the most important of which are first order optimality conditions given by Pontryagin Maximum Principle) and general sufficient optimality conditions (second-order and higher-order). But optimality tests are available only for special classes of problems (linear, linear-quadratic, convex problems). 

\subsection{Smooth manifolds and vector fields}
Here we recall, very briefly, some basic facts of calculus on smooth manifolds, for details consult a regular textbook (e.g., \cite{warner, sternberg}).
\subsubsection{Smooth manifolds}
A $k$-dimensional smooth submanifold $M \subset \mathbb{R}^n$ is defined by one of equivalent ways: 

a) implicitly by a system of regular equations:
\begin{align*}
&f_1 (x) = \dots = f_{n-k} (x) = 0, \quad x \in \mathbb{R}^n, \\ 
&\rank \left(\frac{\partial f_1}{\partial x}, \dots, \frac{\partial f_{n-k}}{\partial x}\right) = n - k,
\end{align*}
%%%%%%%%%%%%%%%%%%% стр. 19 %%%%%%%%%%%%%%%%%%%
or 

b) by a regular parameterization:
\begin{align*}
&x_1 = \Phi_1 (y), \dots, \quad x_n = \Phi_n (y), \qquad y \in \mathbb{R}^k, \quad x  \in \mathbb{R}^n,\\
&\rank \left(\frac{\partial \Phi_1}{\partial y}, \dots, \frac{\partial \Phi_n}{\partial y}\right) = k.
\end{align*}

An abstract smooth $k$-dimensional manifold $M$ (not embedded into $\mathbb{R}^n$) is defined via a system of charts  that agree mutually.

The tangent space to a smooth submanifold $M \subset \mathbb{R}^n$ at a point $x \in M$ is defined as follows for the two above definitions of a submanifold:
\begin{enumerate}[label=(\alph*)]
\item $T_x M = \Ker \frac{\partial f}{\partial x},$
\item $T_x M = \Im \frac{\partial \Phi}{\partial y}.$
\end{enumerate}

Now let $M$ be an abstract smooth manifold. Consider smooth curves $\varphi \colon (-\varepsilon, \varepsilon) \to M$. Then the velocity vector $\dot{\varphi}(0) = \frac{d \varphi}{dt}  (0)$ is defined as the equivalence class
%%%%%%%%%%%%%%%%%%% стр. 20 %%%%%%%%%%%%%%%%%%%
of all smooth curves with $\varphi (0) = q$ and with the same 1-st order Taylor polynomial.

The tangent space to $M$ at a point $q$ is the set of all tangent vectors to $M$ at $q$:
\begin{align*}
T_q M = \{ \dot{\varphi}(0)  \mid \varphi \colon  (-\varepsilon, \varepsilon) \to M \textrm{ smooth,} \quad \varphi (0) = q \}.
\end{align*}
\subsubsection{Smooth vector fields and Lie brackets}
A smooth vector field on $M$ is a smooth mapping
$$ M \ni q \mapsto V(q) \in T_q M.$$
Notation: $V \in \Vec (M)$.

A trajectory of $V$ through a point $q_0 \in M$ is a solution to the Cauchy problem:
\begin{align*}
\dot{q}(t) = V(q(t)), \quad q(0) = q_0.
\end{align*}
Suppose that a trajectory $q(t)$ exists for all times $t \in \mathbb{R}$, then we denote
%%%%%%%%%%%%%%%%%%% стр. 21 %%%%%%%%%%%%%%%%%%%
$e^{tV} (q_0) := q(t)$. The one-parameter group of diffeomorphisms $e^{tV} \colon M \to M$ is the flow of the vector field $V$. 

Consider two vector fields $V,W \in \Vec(M)$. We say that $V$ and $W$ commute if their flows commute:
$$e^{tV} \circ e^{sW} = e^{sW} \circ e^{tV}, \quad t,s \in \mathbb{R}.$$
In the general case $V$ and $W$ do not commute, thus $e^{tV} \circ e^{sW} \ne e^{sW} \circ e^{tV}$, moreover, $e^{tV} \circ e^{tW} \ne e^{tW} \circ e^{tV}$. Thus the curve
$$ \gamma (t) = e^{-tW} \circ e^{-tV} \circ e^{tW} \circ e^{tV} (q) $$
satisfies the inequality $\gamma (t) \ne q$, $t \in \R$. The leading nontrivial term of the Taylor expansion of $\gamma (t)$, $t \to 0$, is taken as the measure of noncommutativity of vector fields $V$ and $W$. Namely, we have:
%%%%%%%%%%%%%%%%%%% стр. 22 %%%%%%%%%%%%%%%%%%%
$\gamma (0) = 0$, $\dot{\gamma} (0) = 0$, $\ddot{\gamma}(0) \ne 0$ generically. Thus the commutator (Lie bracket) of vector fields $V, W$ is defined as 
$$ [V, W] (q) := \frac12 \ddot{\gamma} (0), $$
so that 
$$ \gamma (t) = q + t^2 [V, W] (q) + o(t^2), \qquad t \to 0.$$

\begin{exercise}
Prove that in local coordinates
$$ [V, W] = \frac{\partial W}{\partial x}V - \frac{\partial V}{\partial x}W. $$
\end{exercise}

\paragraph{Example: Reeds-Shepp car}
Consider the vector fields in the right-hand side of the control system
\begin{align*}
&\begin{pmatrix}
\dot{x}\\
\dot{y}\\
\dot{\theta}
\end{pmatrix}
= u
\begin{pmatrix}
\cos \theta \\
\sin \theta \\
0
\end{pmatrix}
 + v
 \begin{pmatrix}
 0 \\
 0 \\
 1
 \end{pmatrix},\\
 &V = \cos \theta \frac{\partial}{\partial x} + \sin \theta \frac{\partial}{\partial y}, \quad W = \frac{\partial}{\partial \theta}.
 \end{align*}
%%%%%%%%%%%%%%%%%%% стр. 23 %%%%%%%%%%%%%%%%%%%
Compute their Lie bracket:
\begin{align*}
[V, W] = \frac{\partial W}{\partial q}V - \frac{\partial V}{\partial q}W = 0 \cdot V - 
 \begin{pmatrix}
0 & 0 & -\sin \theta \\
0 & 0 & \cos \theta \\
0 & 0 & 0 \\
 \end{pmatrix}
 \begin{pmatrix}
 0 \\
 0 \\
 1 
\end{pmatrix}= 
\begin{pmatrix}
\sin \theta \\
-\cos \theta \\
0
\end{pmatrix}.
\end{align*}

There is another way of computing Lie brackets, via commutator of differential operators corresponding to vector fields:
\begin{align*}
[V, W] &= V \circ W - W \circ V = \left(\cos \theta \frac{\partial}{\partial x} + \sin \theta \frac{\partial}{\partial y}\right) \frac{\partial}{\partial \theta} - \frac{\partial}{\partial \theta} \left(\cos \theta \frac{\partial}{\partial x} + \sin \theta \frac{\partial}{\partial y}\right) = \\
&= \sin \theta \frac{\partial}{\partial x} - \cos \theta \frac{\partial}{\partial y}.
\end{align*}
Notice the visual meaning of the vector fields $V, W, [V, W]$ for the car in the plane:
\begin{itemize}
\item $V$ generates the motion forward,
\item $W$ generates rotations of the car,
%%%%%%%%%%%%%%%%%%% стр. 24 %%%%%%%%%%%%%%%%%%%
\item $[V, W]$ generates motion of the car in the direction perpendicular to its orientation, thus physically forbidden.
\end{itemize}
Choosing alternating motions of the car: forward $\to$ rotation counterclockwise $\to$ backward $\to$ rotation clockwise, we can move the car infinitesimally in the forbidden direction. So the Lie bracket $[V, W]$ is generated by a car during parking maneuvers in a limited space.

\subsection{Exercises}
\begin{enumerate}
\item Describe $A_{q_0}$ for Examples 1--5. Which of these systems is controllable?
\item Describe in Example 6:
$$ \Lie_{q_0} (X_1, X_2) = \spann (X_1 (q), X_2 (q), [X_1, X_2](q), [X_1, [X_1, X_2]](q), [X_2, [X_1, X_2]] (q), \dots),$$
where $X_1$ and $X_2$ are vector fields
%%%%%%%%%%%%%%%%%%% стр. 25 %%%%%%%%%%%%%%%%%%%
in the right-hand side of the system:
$$ \dot{q} = u_1 X_1 + u_2 X_2, \quad q \in \mathbb{R}^2 \times \SO(3).$$
\item Show that $S^2$ and $\SO(3)$ are smooth submanifolds. Compute their tangent spaces.
\item Prove in Example 7:
$$ l \to \min \Leftrightarrow J \to \min.$$
\end{enumerate}
%%%%%%%%%%%%%%%%%%% стр. 26 %%%%%%%%%%%%%%%%%%%
\section{Controllability}
In this section we present some basic facts on the controllability problem. The central result is the Orbit theorem, see Th. \ref{th:orbit}. 
\subsection{Controllability of linear systems}
We start from the simplest class of control systems, quite popular in applications.

Linear control systems have the form
\begin{align}
\label{LS}
&\dot{x} = Ax + \sum^k_{i=1} u_i b_i = Ax + Bu,\\
\nonumber
&x \in \mathbb{R}^n, \quad u \in \mathbb{R}^k, \quad u \in L^1 ([0, T], \mathbb{R}^k).
\end{align}
It is easy to find solutions to such systems by the variation of constants method:
\begin{align*}
&x = e^{At}C, \quad C = C(t), \\
%%%%%%%%%%%%%%%%%%% стр. 27 %%%%%%%%%%%%%%%%%%%
&\dot{x} = A e^{At}C + e^{At}\dot{C} = A e^{At}C + Bu, \\
&\dot{C} = e^{-At}Bu(t), \\
&C = \int^t_0 e^{-As} Bu(s)ds +C_0, \\
&x = e^{At} \left(\int^t_0 e^{-As} Bu(s)ds +C_0 \right), \\
&x(0) = C_0 = x_0, \\
&x(t) = e^{At}\left(x_0 + \int^t_0 e^{-As} Bu(s)ds \right).
\end{align*}
Here $e^{At} = \Id + At + \frac{A^2 t^2}{2!} + \dots + \frac{A^n t^n}{n!} + \dots$ is the matrix exponential.
\begin{definition}
A linear system \eqref{LS} is called controllable from a point $x_0 \in \R^n$ for time $T > 0$ (for time not greater than $T$) if 
\begin{align*}
 & A_{x_0} (T) = \mathbb{R}^n 
\qquad ( \textrm{resp. }   A_{x_0} ( \le T) = \mathbb{R}^n).
 \end{align*}
\end{definition}

\begin{theorem}[Kalman controllability test]
Let $T > 0$ and $x_0 \in \R^n$.
Linear system $(\ref{LS})$ is controllable from $x_0$ for time $T$ iff
\begin{align}
\label{Kalman}
\spann (B, AB, \dots, A^{n - 1} B) = \mathbb{R}^n.
\end{align}
%%%%%%%%%%%%%%%%%%% стр. 28 %%%%%%%%%%%%%%%%%%%
\end{theorem}
\begin{proof}
The mapping $u(\cdot) \in L^1 \mapsto x (T) \in \mathbb{R}^n$ is affine, thus its image $A_{x_0} (T)$ is an affine subspace of $\mathbb{R}^n$. Further we rewrite the controllability condition:
\begin{align*}
A_{x_0} (T) = \mathbb{R}^n &\Leftrightarrow \Im~e^{AT} \left(x_0 + \int^T_0 e^{-At} Budt \right) = \mathbb{R}^n \Leftrightarrow \\
&\Leftrightarrow \Im  \int^T_0 e^{-At} Bu(t)dt = \mathbb{R}^n.
\end{align*}

Now we prove the necessity. Let $A_{x_0} (T) = \mathbb{R}^n,$ but $\spann (B, AB, \dots, A^{n-1}B) \ne \mathbb{R}^n$. Then there exists a covector $0 \ne p \in \mathbb{R}^{n*}$ such that 
$$ p A^i B = 0, \quad i = 0, \dots, n-1.$$
By the Cayley-Hamilton theorem, $A^n = \sum^{n-1}_{i=0} \alpha_i A^i$ for some $\alpha_i \in \mathbb{R}$. Thus 
$$ A^m =  \sum^{n-1}_{i=0} \beta^m_i A^i, \quad \beta^m_i \in \mathbb{R}, \quad m \in \mathbb{N}. $$
Consequently,
\begin{align*}
&p A^m B = \sum^{n-1}_{i=0} \beta^m_i p A^i B = 0, \qquad m \in \N,\\
%%%%%%%%%%%%%%%%%%% стр. 29 %%%%%%%%%%%%%%%%%%%
&p e^{-A} B = p \sum^{\infty}_{m=0} \frac{(-A)^m}{m!}B = 0,
\end{align*}
and $\Im \int^T_0 e^{-At} Bu(t)dt \ne \mathbb{R}^n$, contradiction. 

Then we prove the sufficiency. Let $\spann (B, 
AB, \dots, A^{n - 1} B) = \mathbb{R}^n$, but $\Im \int^T_0 e^{-At} Budt \ne \mathbb{R}^n$. Then there exists a covector $0 \ne p \in \mathbb{R}^{n*}$ such that 
\begin{align*}
p \int^T_0 e^{-At} Bu(t)dt = 0 \qquad \forall u \in L^1 ([0, T], \mathbb{R}^k).
\end{align*}
Let $e_1, \dots, e_k$ be the standard frame in $\mathbb{R}^k$. Define the following controls:
\begin{align*}
u(t) = \left\{
\begin{array}{ll}
e_i, & t \in [0, \tau],\\
0, & t \in [\tau, T].
\end{array}
\right.
\end{align*}
We have
\begin{align*}
\int^T_0 e^{-At} Bu(t)dt  = \int^{\tau}_0 e^{-At} b_i dt = \frac{\Id - e^{-A\tau}}{A} b_i,
\end{align*}
thus
\begin{align}
\label{pId0}
p \frac{\Id - e^{-A\tau}}{A} B = 0,
\end{align}
where
%%%%%%%%%%%%%%%%%%% стр. 30 %%%%%%%%%%%%%%%%%%%
\begin{align*}
\frac{\Id - e^{-A\tau}}{A} = - (-\tau \Id + \tau^2 A - \dots + \frac{(-\tau)^m}{(m - 1)!} A^{m - 1} + \dots).
\end{align*}
We differentiate successively identity \eqref{pId0} at $\tau = 0$ and obtain
$$ pB = pAB = \dots = pA^{n - 1} B = 0, $$
thus $\spann (B, AB, \dots, A^{n - 1} B) \ne \mathbb{R}^n$, contradiction.
\end{proof}
Condition \eqref{Kalman} is called Kalman controllability condition.
\begin{corollary}
The following conditions are equivalent:
\begin{itemize}
\item Kalman controllability condition \eqref{Kalman},
\item $\forall t > 0 \ \forall x_0 \in \R^n$ linear system \eqref{LS} is controllable from $x_0$ for time $t$,
\item $\forall t > 0 \ \forall x_0 \in \R^n$ linear system \eqref{LS} is controllable from $x_0$ for time not greater than $t$,
\item $\exists t > 0 \ \exists x_0 \in \R^n$ linear system \eqref{LS} is controllable from $x_0$ for time $t$,
\item $\exists t > 0 \ \exists x_0 \in \R^n$ linear system \eqref{LS} is controllable from $x_0$ for time not greater than $t$.
\end{itemize}
\end{corollary}
In these cases linear system \eqref{LS} is called controllable.

\begin{remark}
For linear systems, controllability for the class of admissible controls $u(\cdot) \in L^1$ is equivalent to controllability for any class of admissible controls $u(\cdot) \in L$ where $L$ is a linear subspace of $L^1$ containing piecewise constant functions. 
\end{remark}
%%%%%%%%%%%%%%%%%%% стр. 31 %%%%%%%%%%%%%%%%%%%
\subsection{Local controllability of nonlinear systems}
Consider now a nonlinear system
\begin{align}
\label{NS}
\dot{x} = f(x, u), \quad x \in \mathbb{R}^n, \quad u \in U \subset \mathbb{R}^m.
\end{align}
Admissible controls are $u(\cdot) \in L^{\infty}([0,T],U)$.

A point $(x_0, u_0) \in \mathbb{R}^n \times U$  is called an equilibrium point of system \eqref{NS} if $f(x_0, u_0) = 0$. We will suppose that
\begin{align}
\label{u0int}
u_0 \in \intt~U
\end{align}
and consider the linearization of system \eqref{NS} at the equilibrium point $(x_0, u_0)$:
\begin{align}
\label{LNS}
&\dot{y} = Ay +Bv, \quad y \in \mathbb{R}^n, \quad v \in \mathbb{R}^m, \\
\nonumber
&A = \frac{\partial f}{\partial x}|_{(x_0, u_0)}, \quad B = \frac{\partial f}{\partial u}|_{(x_0, u_0)}.
\end{align}
It is natural to expect that global properties of linearization \eqref{LNS} 
%%%%%%%%%%%%%%%%%%% стр. 32 %%%%%%%%%%%%%%%%%%%
imply the corresponding local properties of nonlinear system \eqref{NS}. Indeed, there holds the following statement.
\begin{theorem}[Linearization principle for controllability]
If linearization \eqref{LNS} is controllable at an equilibrium point $(x_0, u_0)$ with \eqref{u0int}, then nonlinear system \eqref{NS} satisfies the property:
\begin{align*}
\forall \, T > 0 \quad x_0 \in \intt A_{x_0} (T).
\end{align*}
The more so, nonlinear system is STLC at $x_0$.
\end{theorem}
\begin{proof} Fix any $T > 0$. Let $e_1, \dots, e_n$ be the standard frame in $\mathbb{R}^n$. Since linear system \eqref{LNS} is controllable, then
\begin{align}
\label{y0T}
\forall i = 1, \dots, n \quad \exists v_i \in L^{\infty} ([0, T], \mathbb{R}^m): \quad y_{v_i} (0) = 0, \quad y_{v_i} (T) = e_i.
\end{align}
Construct the following family of controls:
%%%%%%%%%%%%%%%%%%% стр. 33 %%%%%%%%%%%%%%%%%%%
\begin{align*}
u(z, t) = u_0 + z_1 v_1 (t) + \dots + z_n v_n (t), \quad z \in \mathbb{R}^n.
\end{align*}
By condition \eqref{u0int}, for sufficiently small $|z|$ the control $u(z, t) \in U$, thus it is admissible for nonlinear system \eqref{NS}. Consider the corresponding family of trajectories of \eqref{NS}:
\begin{align*}
x(z, t) = x_{u(z, t)}(t), \quad x(z, 0) = x_0, \quad z \in \mathbb{R}^n.
\end{align*}
Let $B$ be a small open ball in $\mathbb{R}^n$ centered at the origin. Since
$$ x(z, T) \in A_{x_0}(T), \quad z \in B, $$
then the mapping
$$ F \colon z \mapsto x(z, T), \quad B \to \mathbb{R}^n $$ satisfies the inclusion
$$ F(B) \subset A_{x_0}(T).$$
It remains to show that $x_0 \in \intt F(B)$. To this end define the matrix function
$$ W(t) = \frac{\partial x (z, t)}{\partial z} |_{z=0}.$$
%%%%%%%%%%%%%%%%%%% стр. 34 %%%%%%%%%%%%%%%%%%%
We show that $\det W(T) = \frac{\partial F}{\partial z} |_{z=0}\ne 0$. This would imply $x_0 = F(0) \in \intt F(B) \subset A_{x_0}(T)$.

Differentiating the identity $\frac{\partial x}{\partial t} = f(x, u(z, t))$ w.r.t. $z$, we get
$$ \frac{\partial}{\partial t} \frac{\partial x}{\partial z}|_{z=0} = \frac{\partial f}{\partial x}|_{(x_0, u_0)} \frac{\partial x}{\partial z}|_{z=0} + \frac{\partial f}{\partial u}|_{(x_0, u_0)} \frac{\partial u}{\partial z}|_{z=0} $$
since $u(0, t) \equiv u_0$ and $x(0, t) \equiv x_0$. Thus we get a matrix ODE
\begin{align}
\label{Wdot}
\dot{W}(t) = AW(t) + B(v_1 (t), \dots, v_n (t))
\end{align}
with the initial condition
$$ W(0) = \frac{\partial x (z, 0)}{\partial z}|_{z=0} = \frac{\partial x_0}{\partial z}|_{z=0} = 0. $$
ODE \eqref{Wdot} means that columns of the matrix $W(t)$ are solutions to linear system \eqref{LNS} with the control $v_i (t)$. By condition \eqref{y0T} we have $W(T) = (e_1, \dots, e_n)$, so $\det W(T) = 1 \ne 0$.

By implicit function theorem, we have $x_0 \in \intt F(B)$, thus $x_0 \in \intt A_{x_0}(T)$.
%%%%%%%%%%%%%%%%%%% стр. 35 %%%%%%%%%%%%%%%%%%%
\end{proof}
\subsection{Orbit theorem}
Let $\mathcal{F} \subset \Vec(M)$ be an arbitrary family of smooth vector fields. 
We assume for simplicity that all vector fields in $\mathcal F$ are complete, i.e., have trajectories defined for any real time.
The attainable set of the family $\mathcal{F}$ from a point $q_0 \in M$ is defined as
$$ A_{q_0} = \{ e^{t_N f_N} \circ \dots \circ e^{t_1 f_1} (q_0) \mid t_i \ge 0, \quad f_i \in \mathcal{F}, \quad N \in \mathbb{N}\}. $$
If we parameterize $\mathcal{F}$ by a control parameter $u$, such attainable set corresponds to piecewise constant controls and arbitrary nonnegative times.

Before studying attainable set, we consider a bigger set --- the orbit of the family $\mathcal{F}$ through the point $q_0$:
$$ O_{q_0} = \{ e^{t_N f_N} \circ \dots \circ e^{t_1 f_1} (q_0) \mid t_i \in \mathbb{R}, \quad f_i \in \mathcal{F}, \quad N \in \mathbb{N}\}. $$
In attainable set we can move only forward along vector fields $f_i \in \mathcal{F}$, while in orbit the backward motion along $f_i$ is also possible, thus 
$$ A_{q_0} \subset O_{q_0}. $$
%%%%%%%%%%%%%%%%%%% стр. 36 %%%%%%%%%%%%%%%%%%%
There hold the following non-trivial relations between attainable sets and orbits:
\begin{enumerate}
\item $O_{q_0}$ has a simpler structure than $A_{q_0}$,
\item $A_{q_0}$ has a reasonable structure inside $O_{q_0}$,
\end{enumerate}
we clarify these relations in the Orbit Theorem and in Krener's theorem. Before that we recall two important constructions.
\paragraph{Action of diffeomorphisms on tangent vectors and vector fields}
Let $M$, $N$ be smooth manifolds, $q \in M$, and let $v \in T_q M$ be a tangent vector. Let $F \colon M \to N$ be a smooth mapping. Then the action (push-forward) of the mapping $F$ on the vector $v$ is defined as follows. Let $\varphi \colon (- \varepsilon, \varepsilon) \to M$ be a smooth curve such that $\varphi (0) = q$, $\dot{\varphi}(0) = v$. Then the tangent vector $F_{*q} v \in T_{F(q)} N$ is defined as $F_{*q}v = \frac{d}{dt}|_{t=0} \quad F \circ \varphi (t).$
%%%%%%%%%%%%%%%%%%% стр. 37 %%%%%%%%%%%%%%%%%%%

Now let $V \in \Vec (M)$ be a smooth vector field, and let $F \colon M \to N$ be a diffeomorphism. Then the vector field $F_{*}V \in \Vec(N)$ is defined by the equality
$$ F_{*}V|_{F(q)} = \frac{d}{dt}|_{t=0} \quad F \circ  e^{tV}(q) = F_{*q} V(q). $$
\paragraph{Immersed submanifolds}
\begin{definition}
A subset $W$ of a smooth manifold $M$ is called a $k$-dimensional immersed submanifold of $M$ if there exists a $k$-dimensional manifold $N$ and a smooth mapping $F \colon N \to M$ such that:
\begin{itemize}
\item $F$ is injective,
\item $\Ker F_{*q} = 0$ for any $q \in N$,
\item $W = F(N)$.
\end{itemize}
\end{definition}
\paragraph{Example 1: Figure 8}
Prove that the curve 
$$ \left\{ x = \sin 2 \varphi \cos \varphi, \quad y = \sin 2 \varphi \sin \varphi  \mid  \varphi \in \left(-\frac{\pi}{2}, \frac{\pi}{2}\right) \right\}$$
%%%%%%%%%%%%%%%%%%% стр. 38 %%%%%%%%%%%%%%%%%%%
is a 1-dimensional immersed submanifold of the 2-dimensional plane.
\paragraph{Example 2: Irrational winding of torus}
Consider the two-dimensional torus $\T^2 = \mathbb{R}^2_{x,y} / \mathbb{Z}^2$, and consider a vector field on it with constant coefficients: $V = p \frac{\partial}{\partial x} + q \frac{\partial}{\partial y}\in \Vec(\T^2)$, $p^2 + q^2 \neq 0$. The orbit of the vector field $V$ through the origin $0 \in \T^2$ may have two different qualitative types:
\begin{itemize}
\item[(1)] $p/q\in\mathbb{Q}\cup \{\infty\}$. Then the orbit of $V$ is closed: $\cl~O_0 = O_0$.
\item[(2)] $p/q\in\mathbb{R} \backslash \mathbb{Q}$. Then the orbit is dense in the torus: $\cl~ O_0 = \T^2$. In this case the orbit $O_0$ is called the irrational winding of the torus.
\end{itemize}
So even for one vector field the
%%%%%%%%%%%%%%%%%%% стр. 39 %%%%%%%%%%%%%%%%%%%
orbit may be an immersed submanifold, but not an embedded submanifold: the topology of the orbit induced by the inclusion $O_0 \subset \mathbb{R}^2$ is weaker than the topology of the orbit induced by the immersion
$$ t \mapsto e^{tV} (0), \quad \mathbb{R} \to O_0. $$

Now we can state the Orbit Theorem.

\begin{theorem}[Orbit Theorem, Nagano-Sussmann]\label{th:orbit}
Let $\mathcal{F} \subset \Vec~ M$, and let $q_0 \in M$. 
\begin{enumerate}
\item $O_{q_0}$ is a connected immersed submanifold of $M$.
\item For any $q \in O_{q_0}$
\begin{align*}
&T_q O_{q_0} = (\mathcal{P}_{*} \mathcal{F})(q) = \{ (P_* V)(q) \mid P \in G, \quad V \in \mathcal{F} \},\\
&G = \{ e^{t_N f_N} \circ \dots \circ e^{t_1 f_1} \mid t_i \in \mathbb{R}, \quad f_i \in \mathcal{F}, \quad N \in \mathbb{N} \}.
 \end{align*}
 \end{enumerate}
\end{theorem}
%%%%%%%%%%%%%%%%%%% стр. 40 %%%%%%%%%%%%%%%%%%%
A proof of the Orbit Theorem is given in \cite{notes}. Below we prove several its important corollaries.
\begin{corollary}
For any $q_0 \in M$ and $q \in O_{q_0}$
\begin{align}
\label{LieqF}
\Lie_q (\mathcal{F}) \subset T_q O_{q_0},
\end{align}
where
$$ \Lie_q (\mathcal{F}) = \spann \{ [f_N, [\dots, [f_2, f_1] \dots]](q) \mid f_i \in \mathcal{F}, N \in \mathbb{N} \} \subset T_q M. $$
\end{corollary}
\begin{proof}
Let $q_0 \in M$, $q \in O_{q_0}$. Take any $f \in \mathcal{F}$. Then $\varphi (t) = e^{t f}(q) \in O_{q_0}$, thus 
$$ \dot{\varphi} (0) = f(q) \in T_q O_{q_0}. $$
It follows that $\mathcal{F}(q) \subset T_q O_{q_0}$.
%%%%%%%%%%%%%%%%%%% стр. 41 %%%%%%%%%%%%%%%%%%%

Further, take any $f_1, f_2 \in \mathcal{F}$, then $\varphi(t) = e^{-t f_2} \circ e^{-t f_1} \circ e^{t f_2} \circ e^{t f_1} (q) \in O_{q_0}$. Thus
$$ \frac{d}{dt}|_{t=0}~\varphi(\sqrt{t}) = [f_1, f_2] (q) \in T_q O_{q_0}. $$
It follows that $[\mathcal{F}, \mathcal{F}] (q) \subset T_q O_{q_0}$.

We prove similarly that $[[\mathcal{F}, \mathcal{F}], \mathcal{F}] (q) \subset T_q O_{q_0}$, and by induction that $\Lie_q (\mathcal{F}) \subset T_q O_{q_0}$.
\end{proof}
In the analytic case inclusion \eqref{LieqF} turns into equality.
\begin{prop}
Let $M, \mathcal{F}$ be real-analytic. Then for any $q_0 \in M$ and $q \in O_{q_0}$
$$ \Lie_q (\mathcal{F}) = T_q O_{q_0}. $$
\end{prop}
This proposition is proved in \cite{notes}.
%%%%%%%%%%%%%%%%%%% стр. 42 %%%%%%%%%%%%%%%%%%%
But in a smooth non-analytic case inclusion \eqref{LieqF} may become strict.
\paragraph{Example: Orbit of non-analytic system}
Let $M = \mathbb{R}^2_{x,y}$, $\mathcal{F} = \{ f_1, f_2 \}$, $f_1 = \frac{\partial}{\partial x}$, $f_2 =  a(x)\frac{\partial}{\partial y}$, where $a \in C^\infty (\mathbb{R})$, $a(x) = 0$ for $x \le 0$, $a(x) > 0$ for $x > 0$.

It is easy to see that $O_q = \mathbb{R}^2$ for any $q \in \mathbb{R}^2$. Although, for $x \le 0$ we have 
$$ \Lie_q (\mathcal{F}) = \spann (f_1 (q)) \ne T_q O_q. $$

%%%%%%%%%%%%%%%%%%%%%%%%%%%%%%%%%%%%%%%%%%%%%%%%%%%%%%%%%%%%%%%
%end of part1
%%%%%%%%%%%%%%%%%%%%%%%%%%%%%%%%%%%%%%%%%%%%%%%%%%%%%%%%%%%%%%%
%%%%%%%%%%%%%%%%%%% стр. 43 %%%%%%%%%%%%%%%%%%%
\subsection{Frobenius theorem}

A distribution on a smooth manifold $M$ is a smooth mapping:
$$ \Delta \colon q \mapsto \Delta_q \subset T_q M, \quad q \in M, $$
where the subspaces $\Delta_q$ have the same dimension called the rank of $\Delta$. 

An immersed submanifold $N \subset M$ is called an integral manifold of $\Delta$ if
$$ \forall q \in N \quad T_q N = \Delta_q. $$
A distribution $\Delta$ on $M$ is called integrable if for any point $q \in M$ there exists an integral manifold $N_q \ni q.$

Denote by 
$$ \bar{\Delta} = \{ f \in \Vec(M) \mid  f(q) \in \Delta_q \quad \forall q \in M \} $$
the set of vector fields tangent to $\Delta$. 

A distribution $\Delta$ is called holonomic if $[\bar{\Delta}, \bar{\Delta}] \subset \bar{\Delta}$.

\begin{theorem}[Frobenius]
A distribution is integrable iff it is holonomic.
%%%%%%%%%%%%%%%%%%% стр. 44 %%%%%%%%%%%%%%%%%%%
\end{theorem}
\begin{proof}
Necessity. Take any $f,  g \in \bar{\Delta}$. Let $q \in M$, and let $N_q \ni q$ be the integral manifold of $\Delta$ through $q$. Then
$$ \varphi(t)= e^{-tg} \circ e^{-tf}\circ e^{tg} \circ e^{tf} (q) \in N_q, $$
thus 
$$ \frac{d}{dt}\left|_{t=0} \varphi(\sqrt{t}) = [f, g](q) \in T_q N_q = \Delta_q.\right. $$
So $ [f, g] \in \bar{\Delta}$, and the inclusion $[\bar{\Delta}, \bar{\Delta}] \subset \bar{\Delta}$ follows.

Sufficiency. We consider only the analytic case. We have $[\bar{\Delta}, \bar{\Delta}] \subset \bar{\Delta}, [[\bar{\Delta}, \bar{\Delta}], \bar{\Delta}] \subset [\bar{\Delta}, \bar{\Delta}] \subset \bar{\Delta}$, and inductively $\Lie_q (\bar{\Delta}) \subset \bar{\Delta}_q = \Delta_q.$ The reverse inclusion is obvious, thus $\Lie_q (\bar{\Delta}) = \Delta_q$, $q \in M$. Denote $N_q = O_q (\bar{\Delta})$ and prove that $N_q$ is an integral manifold of $\Delta$:
$$ T_{q'} N_q = T_{q'} (O_q (\bar{\Delta})) = \Lie_{q'} (\bar{\Delta}) = \Delta_{q'}, \quad q' \in N_q. $$
So $N_q \ni q$ is the integral manifold of $\Delta$, and $\Delta$ is integrable.
%%%%%%%%%%%%%%%%%%% стр. 45 %%%%%%%%%%%%%%%%%%%
\end{proof}
Consider a local frame of $\Delta$:
$$ \Delta_q = \spann (f_1 (q), \dots, f_k (q)), \quad q \in S \subset M, \quad f_1 (q), \dots, f_k (q) \in \Vec(S), $$
where $S$ is an open subset of $M$. Then the inclusion $[\bar{\Delta}, \bar{\Delta}] \subset \bar{\Delta}$ takes the form
$$ [f_i, f_j] (q) = \sum^k_{e=1} c^l_{ij}(q) f_l(q), \qquad q \in S, \quad c_{ij}^l \in C^{\infty}(S).$$
This equality is called Frobenius condition.
\subsection{Rashevsky-Chow theorem}
A system $\mathcal{F} \subset \Vec(M)$ is called completely nonholonomic (full-rank, bracket-generating) if $\Lie_q (\mathcal{F}) = T_q M$ for any $q \in M$.
\begin{theorem}[Rashevsky-Chow]
If $\mathcal{F} \subset \Vec(M)$  is completely nonholonomic and $M$ is connected, then $O_q = M$ for any $q \in M$.
\end{theorem}
\begin{proof}
Take any $q \in M$ and any $q_1 \in O_q$. We have
%%%%%%%%%%%%%%%%%%% стр. 46 %%%%%%%%%%%%%%%%%%%
$T_{q_1} O_q \supset \Lie_{q_1} (\mathcal{F}) = T_{q_1} M$, thus $\dim O_q = \dim M$, i.e., $O_q$ is open in $M$.

On the other hand, any orbit is closed as a complement to the union of all other orbits.

Thus any orbit is a connected component of $M$. Since $M$ is connected, each orbit coincides with $M$.
\end{proof}
\subsection{Attainable sets of full-rank systems}
Let $\mathcal{F} \subset \Vec(M)$ be a full-rank system. The assumption of full rank is not very restrictive in the analytic case: if it is violated, we can consider the restriction of $\mathcal{F}$ to its orbit, and this restriction is full-rank.

What is the possible structure of attainable sets of $\mathcal{F}$? It is easy to construct systems in the two-dimensional plane that have the following attainable sets:
\begin{itemize}
\item smooth full-dimensional manifold without boundary,
%%%%%%%%%%%%%%%%%%% стр. 47 %%%%%%%%%%%%%%%%%%%
\item smooth full-dimensional manifold with smooth boundary,
\item smooth full-dimensional manifold with non-smooth boundary, with corner or cusp singularity.
\end{itemize}
But it is impossible to construct attainable set that is:
\begin{itemize}
\item a lower-dimensional submanifold,
\item a set whose boundary points are isolated from its interior points. These possibilities are forbidden respectively by items (1) and (2) of the following theorem.
\end{itemize}
\begin{theorem}[Krener]
Let $\mathcal{F} \subset \Vec(M)$, and let $\Lie_q \mathcal{F} = T_q M$ for any $q \in M$. Then:
\begin{enumerate}
\item[$(1)$] $\intt A_q \ne \varnothing$ for any $q \in M$,
\item[$(2)$] $\cl (\intt A_q) \supset A_q$ for any $q \in M$.
\end{enumerate}
\end{theorem}

\begin{proof}
Since item (2) implies item (1), we prove item (2).

We argue by induction on dimension of $M$. If $\dim M = 0$, there is nothing to prove. Let $\dim M > 0$.

Take any $q_1 \in A_q$, and fix any neighborhood $q_1 \in W(q_1) \subset M$. We show that $\intt A_q \cap W(q_1) \ne \varnothing$.
%%%%%%%%%%%%%%%%%%% стр. 48 %%%%%%%%%%%%%%%%%%%
There exists $f_1 \in \mathcal{F}$ such that $f_1 (q_1) \ne 0,$ otherwise $\mathcal{F}(q_1) = \{0\} = \Lie_{q_1} (\mathcal{F}) = T_{q_1} M$, a contradiction. Consider the following set for small $\varepsilon_1 > 0$:
$$ N_1 = \{ e^{t_1f_1} (q_1) \mid 0 < t_1 < \varepsilon_1 \} \subset W(q_1) \cap A_q.$$
$N_1$ is a smooth 1-dimensional manifold. If $\dim M = 1$, then $N_1$ is open, thus $N_1 \subset \intt A_q$, so $\intt A_q \cap W(q_1) \ne \varnothing$. Since the neighborhood $W(q_1)$ is arbitrary, $q_1 \in \cl (\intt A_q)$.

Let $\dim M > 1$. There exist $q_2 = e^{t_1^1f_1}(q_1) \in N_1 \cap W(q_1)$ and $f_2 \in \mathcal{F}$ such that $f_2 (q_2) \not\in T_{q_2} N_1$. Otherwise $\dim \mathcal{F}(q_2) = \dim \Lie_{q_2} (\mathcal{F}) = T_{q_2} M = 1$ for any $q_2 \in N_2 \cap W$, and $\dim M = 1$. Consider the following set for small $\varepsilon_2$:
$$ N_2 = \{ e^{t_2 f_2} \circ e^{t_1 f_1} (q_2)  \mid t_1^1 < t_1 < t_1^1 + \varepsilon_2, \  0 < t_2 < \varepsilon_2 \} \subset W(q_1) \cap A_q. $$
$N_2$ is a smooth 2-dimensional manifold.
If $\dim M = 2$, then $N_2$ is open, thus $N_2 \subset \intt A_q \cap W(q_1) \ne \varnothing$ and $q_1 \in \cl (\intt A_q)$. 

If $\dim M > 2$, we proceed by induction.
%%%%%%%%%%%%%%%%%%% стр. 49 %%%%%%%%%%%%%%%%%%%
\end{proof}
\subsection{Exercises}
\begin{enumerate}
\item For the system modeling stopping of a train, prove that $O_{x_0} = \mathbb{R}^2$ and $A_{x_0} = \mathbb{R}^2$ for any $x_0 \in \mathbb{R}^2$.
\item For the Markov-Dubins car, prove that:
\begin{itemize}
\item $O_{q_0} =  \mathbb{R}^2 \times S^1$ for any $q_0 \in \mathbb{R}^2 \times S^1$,
\item $A_{q_0} =  \mathbb{R}^2 \times S^1$ for any $q_0 \in \mathbb{R}^2 \times S^1$ (hint: use periodicity of the vector fields $X_0 \ne X_1$, $X_0 = \cos \theta \frac{\partial}{\partial x} + \sin \theta \frac{\partial}{\partial y}$, $X_1 =  \frac{\partial}{\partial \theta}$).
\end{itemize}
\end{enumerate}
\section{Optimal control problems}
\subsection{Problem statement}
We consider the following optimal control problem:
\begin{align}
&\dot{q} = f(q, u), \quad q \in M, \quad u \in U \subset \mathbb{R}^m, \label{p21}\\
&q(0) = q_0, \quad q(t_1) = q_1, \label{p22}\\
&J = \int^{t_1}_0 \varphi(q, u) dt \to \min, \label{p23}\\
&t_1 \textrm{ fixed or free.}\nonumber
\end{align}
%%%%%%%%%%%%%%%%%%% стр. 50 %%%%%%%%%%%%%%%%%%%
The following assumptions are supposed for dynamics $f(q, u)$:
\begin{itemize}
\item $q \mapsto f(q, u)$ smooth for any $u \in U,$
\item $(q, u) \mapsto f(q, u)$ continuous for any $q \in M$, $u \in \cl(U)$,
\item $(q, u) \mapsto \frac{\partial f}{\partial q} (q, u)$ continuous for any $q \in M$, $u \in \cl(U)$.
\end{itemize}
The same assumptions are supposed for the function $\varphi(q, u)$ that determines the cost functional~$J$.

Admissible control is $u \in L^{\infty} ([0, t_1], U)$.

\subsection{Existence of optimal controls}
\begin{theorem}[Filippov]
Let $U \subset \mathbb{R}^m$ be compact.

Suppose that the set $\{ (f(q, u), \varphi(q, u)) \mid u \in U \}$ is convex for any $q \in M$.

Suppose that there exists a compact $K \subset M$ such that $f(q, u) = 0$, $\varphi (q, u) = 0$ for any $u \in U$, $q \in M \backslash K$.

Then optimal control exists for any
%%%%%%%%%%%%%%%%%%% стр. 51 %%%%%%%%%%%%%%%%%%%
$q_0 \in M$ and any $q_1 \in A_{q_0} (t_1)$ for the optimal control problem \eqref{p21}--\eqref{p23} with fixed~$t_1$. 
\end{theorem}

\begin{remark}
Suppose that there exists an apriori bound $A_{q_0} (t_1) \subset B$, where $B \subset M$ is a compact. Take a compact $K \supset \intt K \supset B$ and a function $g \in C^{\infty} (M)$ such that $g|_B \equiv 1,~g|_{M\backslash K} \equiv 0$. Consider a new problem 
\begin{align*}
&\dot{q} =\tilde{f}(q, u) = f(q, u) \cdot g(q),\\
&\tilde{J} = \int^{t_1}_0 \tilde{\varphi} (q, u) dt \to \min,\quad \tilde{\varphi} (q, u) = \varphi (q, u) \cdot g(q).
\end{align*}
Then the new problem satisfies the third condition of Filippov theorem and has the same solution as the initial problem. Thus, when applying Filippov theorem, we can replace its third condition by an apriori estimate of attainable set.
\end{remark}

Now consider a time-optimal problem
\begin{align}
&\dot{q} = f(q, u), \quad q \in M, \quad u \in U \subset \mathbb{R}^m, \label{pt21}\\
&q(0) = q_0, \quad q(t_1) = q_1, \label{pt22}\\
& t_1  \to \min. \label{pt23}
\end{align}

\begin{theorem}[Filippov]
Let $U \subset \mathbb{R}^m$ be compact.

Suppose that the set $\{ f(q, u) \mid u \in U \}$ is convex for any $q \in M$.

Suppose that there exists a compact $K \subset M$ such that $f(q, u) = 0$  for any $u \in U$, $q \in M \backslash K$.

Then optimal control exists for any $q_0 \in M$ and any 
  $q_1 \in A_{q_0}$ for the time-optimal problem \eqref{pt21}--\eqref{pt23}. 
\end{theorem}

\begin{remark}
If $\varphi(q,u) \geq C > 0$ for all $q \in M$, $u \in U$, then the optimal control problem \eqref{p21}--\eqref{p23} with free~$t_1$ can be reduced to a time-optimal problem by introducing a new time $\tau$ such that $d \tau = \varphi(q,u) d t$.

If $\varphi(q,u)$ is not bounded from zero from below, then  existence of optimal controls may fail, consider, e.g., the problem with free $t_1$
\begin{align*}
&\dot{x} = u, \quad x, u  \in \R,\\
&x(0) = x_0, \quad x(t_1) = x_1,\\
& \int_0^{t_1} u^2 dt  \to \min. 
\end{align*}
\end{remark}

\subsection{Elements of symplectic geometry}
In order to state a fundamental necessary optimality condition --- Pontryagin Maximum Principle --- we need some basic facts of symplectic geometry, which we review in this subsection.

Let $M$ be an $n$-dimensional smooth manifold. Then the disjoint union of its tangent spaces $\bigsqcup\limits_{q \in M} T_q M = TM$ is called its tangent bundle. If $(x_1, \dots, x_n)$ are local coordinates on $M$, then any tangent vector $v \in T_q M$ has a decomposition $v = \sum^n_{i=1} v_i \frac{\partial}{\partial x_i}$. Thus $(x_1, \dots, x_n;~v_1, \dots, v_n)$ are local coordinates on $TM$, which is thus a $2n$-dimensional smooth manifold.

For any point $q \in M$, the dual space $(T_q M)^* = T^*_q M$ is called the cotangent space to $M$ at $q$. The disjoint union
%%%%%%%%%%%%%%%%%%% стр. 52 %%%%%%%%%%%%%%%%%%%
$\bigsqcup\limits_{q \in M} T^*_q M = T^* M$ is called the cotangent bundle of $M$. If $(x_1, \dots, x_n)$ are local coordinates on $M$, then any covector $\lambda \in T^*_q M$ has  a decomposition $\lambda = \sum^n_{i=1} \xi_i dx_i$. Thus $(x_1, \dots, x_n;~ \xi_1, \dots, \xi_n)$ are local coordinates on $T^* M$ called canonical coordinates. In particular, $T^* M$ is a smooth $2n$-dimensional manifold.

The canonical projection is:
$$ \pi \colon T^* M \to M, \quad T^*_q M \ni \lambda \mapsto q \in M. $$
The Liouville (tautological) 1-form $s \in \Lambda^1 (T^* M)$ acts as follows:
$$ \langle s_\lambda, w \rangle = \langle \lambda, \pi_* w \rangle, \quad \lambda \in T^* M, \quad w \in T_\lambda (T^*M). $$
In canonical coordinates on $T^*M$:
\begin{align*}
&w = \sum^n_{i=1} a_i \frac{\partial}{\partial x_i} + b_i  \frac{\partial}{\partial \xi_i}, \\
&\pi_* w = \sum^n_{i=1} a_i \frac{\partial}{\partial x_i}, \\
&\lambda = \sum^n_{i=1} \xi_i dx_i, \\
%%%%%%%%%%%%%%%%%%% стр. 53 %%%%%%%%%%%%%%%%%%%
&\langle s_\lambda, w \rangle = \sum^n_{i=1} \xi_i a_i, \\
&s_\lambda = \sum^n_{i=1} \xi_i dx_i.
\end{align*}
(In mechanics, the Liouville form is known as $s = pdq = \sum^n_{i=1} p_i dq_i$).

The canonical symplectic structure on $T^*M$ is $\sigma = d s \in \Lambda^2 (T^* M)$. In canonical coordinates $\sigma = \sum^n_{i=1} d \xi_i \wedge dx_i$ (in mechanics $\sigma = dp \wedge dq = \sum^n_{i=1} d p_i \wedge dq_i$).

A Hamiltonian is an arbitrary function $h \in C^\infty (T^* M)$.

The Hamiltonian vector field $\vec{h} \in \Vec(T^* M)$ with the Hamiltonian function $h$ is defined by the equality $dh = \sigma (\cdot, \vec{h})$. In canonical coordinates:
\begin{align*}
&h = h (x, \xi), \\
&dh = h_x d_x + h_\xi  d_\xi = \sum^n_{i=1} \frac{\partial h}{\partial x_i} d_{x_i} + \frac{\partial h}{\partial \xi_i} d \xi_i, \\
&\sigma = d \xi \wedge dx = \sum^n_{i=1} d \xi_i \wedge d {x_i}, \\
&\vec{h} = \frac{\partial h}{\partial \xi} \frac{\partial}{\partial x} - \frac{\partial h}{\partial x} \frac{\partial}{\partial \xi} = \sum^n_{i=1} \frac{\partial h}{\partial \xi_i}\frac{\partial}{\partial x_i} - \frac{\partial h}{\partial x_i} \frac{\partial}{\partial \xi_i} .
\end{align*}
%%%%%%%%%%%%%%%%%%% стр. 54 %%%%%%%%%%%%%%%%%%%
The corresponding Hamiltonian system of ODEs is 
$$ \dot{\lambda} = \vec{h} (\lambda), \quad \lambda \in T^* M. $$
In canonical coordinates:
\begin{equation*}
 \begin{cases}
\dot{x} = \frac{\partial h}{\partial \xi},\\
\dot{\xi} = - \frac{\partial h}{\partial x},
\end{cases}
\end{equation*}
or 
\begin{equation*}
 \begin{cases}
\dot{x}_i = \frac{\partial h}{\partial \xi_i},\\
\dot{\xi}_i = - \frac{\partial h}{\partial x_i}, \quad i = 1, \dots, n.
\end{cases}
\end{equation*}
The Poisson bracket of Hamiltonians $h, g \in C^\infty (T^*M)$ is the Hamiltonian $\{ h, g \} \in C^\infty (T^* M)$ defined by the equalities
$$ \{ h, g \} = \vec{h} g = \sigma (\vec{h}, \vec{g}). $$
In canonical coordinates:
$$  \{ h, g \} = \frac{\partial h}{\partial \xi} \frac{\partial g}{\partial x} - \frac{\partial h}{\partial x} \frac{\partial g}{\partial \xi} = \sum^n_{i=1} \frac{\partial h}{\partial \xi_i}\frac{\partial g}{\partial x_i} - \frac{\partial h}{\partial x_i} \frac{\partial g}{\partial \xi_i}. $$
%%%%%%%%%%%%%%%%%%% стр. 55 %%%%%%%%%%%%%%%%%%%
\begin{lemma}
Let $h, g, k \in C^\infty (T^* M)$, and $\alpha, \beta \in \mathbb{R}$. Then:
\begin{itemize}
\item $\{ \alpha h + \beta g, k \} = \alpha \{ h, k \} + \beta \{ g, k \},$
\item $\{ h, g \} = - \{ g, h \},$
\item $\{ h, h \} = 0,$
\item $\{ h, \{ g, k \} \} + \{ g, \{ k, h \} \} + \{ k, \{ h, g \} \} = 0,$
\item $\{ h, gk \} = \{ h, g \}k + g \{ h, k \}$.
\end{itemize}
\end{lemma}
\begin{corollary}
Let $h, g \in C^\infty (T^* M)$. Then $\overrightarrow{\{ h, g \}} = [\vec{h}, \vec{g}]$.
\end{corollary}
\begin{proof}
Let $h, g, k \in C^\infty (T^* M)$. Then  $[\vec{h}, \vec{g}]k = (\vec{h}\vec{g} - \vec{g}\vec{h})k = \vec{h}\vec{g}k - \vec{g}\vec{h}k = \vec{h}\{ g, k \} - \vec{g}\{ h, k \} = \{ h, \{ g, k \}\} - \{ g, \{ h, k \} \} = \{ h, \{ g, k \} \} + \{ g, \{ k, h \} \} = - \{ k, \{ h, g \} \} = \{ \{ h, g \}, k \} = \overrightarrow{\{ h, g \}} k$.
\end{proof}
%%%%%%%%%%%%%%%%%%% стр. 56 %%%%%%%%%%%%%%%%%%%
\begin{theorem}[N$\ddot{o}$ther]
Let $a, h \in C^\infty (T^* M)$. Then
$$ a (e^{t \vec{h}} (\lambda)) \equiv \const \Leftrightarrow \{ h, a \} = 0. $$
\end{theorem}
\begin{proof}
$a (e^{t \vec{h}} (\lambda)) \equiv \const \Leftrightarrow  \vec{h} a  = 0 \Leftrightarrow \{ h, a \} = 0.$
\end{proof}
Let $X \in \Vec(M)$. The corresponding linear on fibers of $T^* M$ Hamiltonian is defined as follows:
$$ h_X (\lambda) = \langle \lambda, X(q) \rangle, \quad q = \pi (\lambda). $$
In canonical coordinates:
\begin{align*}
&X = \sum^n_{i=1} X_i  \frac{\partial}{\partial x_i}, \\
&h_X (x, \xi) = \sum^n_{i=1} \xi_i X_i.
\end{align*}
\begin{lemma}
Let $X, Y \in \Vec(M)$. Then:
\begin{itemize}
\item $\{ h_X, h_Y \} = h_{[X, Y]},$
\item $\{ \vec{h}_X, \vec{h}_Y \} = \vec{h}_{[X, Y]},$
\item $\pi_* \vec{h}_X = X.$
\end{itemize}
%%%%%%%%%%%%%%%%%%% стр. 57 %%%%%%%%%%%%%%%%%%%
\end{lemma}
\begin{proof}
Computation in canonical coordinates.
\end{proof}
The vector field $\vec{h}_X \in \Vec(T^* M)$ is called the Hamiltonian lift of the vector field $X \in \Vec(M)$.

\subsection{Pontryagin Maximum Principle}
Consider optimal control problem \eqref{p21}--\eqref{p23} with fixed terminal time $t_1$.
\begin{theorem}[PMP]
\label{th:PMP}
If $u(t)$ and $q(t), t \in [0, t_1]$, are optimal, then there exist a curve $\lambda_t \in \Lip ([0, t_1], T^* M)$, $\lambda_t \in T^*_{q(t)} M$, and a number $\nu \le 0$ such that the following conditions hold for almost all $t \in [0, t_1]$:
\begin{enumerate}
\item $\dot{\lambda}_t = \vec{h}^\nu_{u(t)} (\lambda_t)$,
\item $h^\nu_{u(t)} (\lambda_t) = \max\limits_{v \in U} h^\nu_{v} (\lambda_t) $,
%%%%%%%%%%%%%%%%%%% стр. 58 %%%%%%%%%%%%%%%%%%%
\item $(\lambda_t, \nu) \ne (0, 0)$.
\begin{remark}
If the terminal time $t_1$ is free, then the following condition is added to 1--3:
\end{remark}
\item $h^\nu_{u(t)} (\lambda_t) \equiv 0.$ \label{p25}
\end{enumerate}
\end{theorem}
\paragraph{Time-optimal problem}
We have $J = t_1 = \int^{t_1}_0 1 dt \to \min$, and PMP is expressed in terms of the shortened Hamiltonian $g_u (\lambda) = \langle \lambda, f(q, u) \rangle$.
\begin{corollary}
If $u(t)$ and $q(t), t \in [0, t_1]$, are time-optimal, then there exists a curve $\lambda_t \in \Lip ([0, t_1], T^*M)$ for which the following conditions hold for almost all $t \in [0, t_1]$:
\begin{enumerate}
\item $\dot{\lambda}_t = \vec{g}_{u(t)} (\lambda_t)$,
\item $g_{u(t)}  (\lambda_t)  = \max\limits_{v \in U} g^\nu_v  (\lambda_t),$
%%%%%%%%%%%%%%%%%%% стр. 59 %%%%%%%%%%%%%%%%%%%
\item $\lambda_t \ne 0,$
\item $g_{u(t)}  (\lambda_t) \equiv \const \ge 0.$
\setcounter{cntr}{\value{enumi}}
\end{enumerate}
\end{corollary}
\paragraph{Optimal control problem with general boundary conditions}
Consider optimal control problem \eqref{p21}, \eqref{p23}, where the boundary condition \eqref{p22} is replaced by the following more general one:
\begin{align}
\label{p24}
q(0) \in N_0, \quad q(t_1) \in N_1.
\end{align}
Here $N_0, N_1 \subset M$ are smooth submanifolds.

For problem \eqref{p21}, \eqref{p23}, \eqref{p24} there hold Pontryagin Maximum Principle with conditions 1--3 of Th.~\ref{th:PMP} for fixed $t_1$ (plus condition 4 for free $t_1$), with additional transversality conditions
\begin{enumerate}
\setcounter{enumi}{\value{cntr}}
\item $\lambda_0 \perp T_{q_0} N_0, \quad \lambda_{t_1} \perp T_{q (t_1)} N_1$.
\end{enumerate}
%%%%%%%%%%%%%%%%%%% стр. 60 %%%%%%%%%%%%%%%%%%%
A control $u(t)$ and a trajectory $q(t)$ that satisfy PMP are called extremal control and extremal trajectory; a curve $\lambda_t$ that satisfy PMP is called extremal.
\begin{remark}
If a pair $(\lambda_t, \nu)$ satisfy PMP, then for any $k > 0$ the pair $(k \lambda_t, k \nu)$ also satisfies PMP.

The case $\nu < 0$ is called the normal case. In this case the pair $(\lambda_t, \nu)$ can be normalized to get $\nu= -1$.

The case $\nu = 0$ is called the abnormal case.
\end{remark}

\begin{theorem}
Let $H \in C^2(T^*M)$. Then a curve $\lambda_t$ is extremal iff it is a trajectory of the Hamiltonian system $\dot \lambda_t = \vec{H}(\lambda_t)$.
\end{theorem}

\subsection{Solution to optimal control problems}
\paragraph{Stopping a train}
We have the time-optimal problem
\begin{align*}
&\dot{x}_1 = x_2, \quad x = (x_1, x_2) \in \mathbb{R}^2,\\
&\dot{x}_2 = u, \quad |u| \le 1,\\
%%%%%%%%%%%%%%%%%%% стр. 61 %%%%%%%%%%%%%%%%%%%
&x(0) = x^0, \quad x(t_1) = x^1 = (0, 0),\\
&t_1 \to \min.
\end{align*}
The right-hand side of the control system $f(x, u) = (x_2, u)^T$ satisfies the bound
$$ | f (x, u) | = \sqrt{x^2_2 + u^2} \le \sqrt{x^2_2 + 1} \le |x| + 1, $$
thus $r = x^2$ satisfies the differential inequality
$$ \dot{r} = 2 \langle x, \dot{x} \rangle = 2 \langle x, f(x, u) \rangle  \le 2(r + 1). $$
So $r(t) \le e^{2t} (r_0 + 1)$, thus attainable set satisfies the apriori bound
$$ A_{x^0} (\le t) \subset \{ x \in \mathbb{R}^2  \mid  |x| \le e^t \sqrt{(x^0)^2 + 1} \}.$$
Thus we can assume that there exists a compact $K \subset \mathbb{R}^2$ such that the right-hand side of the control system vanishes outside of $K$ (one of conditions of Filippov theorem).
%%%%%%%%%%%%%%%%%%% стр. 62 %%%%%%%%%%%%%%%%%%%

Now we compute the orbit $O_{x^0}$. Denote $\mathcal{F} = \{ f(x, u)  \mid  |u| \le 1 \}$. We have $f(x, 0) = x_2 \frac{\partial}{\partial x_1} \in \mathcal{F}$, $f(x, 1) = x_2 \frac{\partial}{\partial x_1} + \frac{\partial}{\partial x_2} \in \mathcal{F}$, thus $f(x, 1) - f(x, 0) = \frac{\partial}{\partial x_2} \in \Span (\mathcal{F})$.

Consequently, $[x_2 \frac{\partial}{\partial x_1}, \frac{\partial}{\partial x_2}] = - \frac{\partial}{\partial x_1} \in \Lie_x (\mathcal{F})$. Summing up, $\Lie_x (\mathcal{F}) \supset \Span (\frac{\partial}{\partial x_1}, \frac{\partial}{\partial x_2}) = T_x \mathbb{R}^2$, thus $O_{x^0} = \mathbb{R}^2$ for any $x^0 \in \mathbb{R}^2$.

Now we study the attainable set $A_{x^0}$. For the controls $u = \pm 1$, the trajectories are parabolas $x_1 = \pm \frac{x^2_2}{2} + C$. Geometrically it is obvious that $A_{x^0} \ni x' = (0,0)$ for any point $x^0 \in \mathbb{R}^2$.

The set of control parameters $U$ is compact, and the set of admissible velocity vectors $f(x, U)$ is convex for any $x \in \mathbb{R}^2$. All hypotheses of Filippov
%%%%%%%%%%%%%%%%%%% стр. 63 %%%%%%%%%%%%%%%%%%%
theorem are satisfied, thus optimal control exists.

We apply PMP using canonical coordinates on $T^* \mathbb{R}^2$. We decompose a covector $\lambda = \psi_1 dx_1 + \psi_2 dx_2 \in T^* \mathbb{R}^2$, then the shortened Hamiltonian of PMP reads
$$ h_u (\lambda) = \psi_1 x_2 + \psi_2 u, $$
and the Hamiltonian system $\dot{\lambda} = \vec{h}_u (\lambda)$ reads
\begin{align*}
&\dot{x}_1 = x_2, \quad \dot{\psi}_1 = 0, \\
&\dot{x}_2 = u, \quad \dot{\psi}_2 = - \psi_1.
\end{align*}
The maximality condition of PMP has the form
$$ h_u (\lambda) = \psi_1 x_2 + \psi_2 u \to \max\limits_{|u| \le 1},$$
and the nontriviality condition is
$$ (\psi_1 (t), \psi_2 (t)) \ne (0, 0). $$
The Hamiltonian system implies that $\psi_1 \equiv \const$, $\psi_2 (t)$ is linear, moreover, $\psi_2 (t) \not\equiv 0$
%%%%%%%%%%%%%%%%%%% стр. 64 %%%%%%%%%%%%%%%%%%%
with account of the nontriviality condition. The maximality condition yields:
\begin{align*}
&\psi_2 (t) > 0 \Rightarrow u(t) = 1, \\
&\psi_2 (t) < 0 \Rightarrow u(t) = -1.
\end{align*}
Thus extremal trajectories are
$$ x_1 = \pm \frac{x^2_2}{2} + C, $$
and the number of switchings (discontinuities) of control is not greater than 1. Let us draw such trajectories backward in time, starting from the origin $x^1$:
\begin{itemize}
\item the controls $u = \pm 1$, $u = -1$ generate two half-parabolas terminating at $x^1$:
$$ x_1 = \frac{x^2_2}{2}, \quad x_2 \le 0, \textrm{ and } x_1 = -\frac{x^2_2}{2}, \quad x_2 \ge 0, $$
\item denote the union of these half-parabolas as $\Gamma$,
%%%%%%%%%%%%%%%%%%% стр. 65 %%%%%%%%%%%%%%%%%%%
\item after one switching, parabolic arcs with $u = 1$ terminating at the half-parabola $x_1 = -\frac{x^2_2}{2}, \quad x_2 \ge 0$, fill the part of the plane $\mathbb{R}^2$ below the curve $\Gamma$,
\item similarly, after one switching, parabolic arcs with $u = -1$ fill the part of the plane over the curve $\Gamma$.
\end{itemize}

So through each point of the plane $\mathbb{R}^2$ passes a unique extremal trajectory. Taking into account existence of optimal controls, the extremal trajectories are optimal.

The optimal control found has explicit dependence on the current point of the plane:
\begin{itemize}
\item if $x_1 = \frac{x^2_2}{2}, \quad x_2 \le 0$, or if the point $(x_1, x_2)$ is below the curve $\Gamma$, then $u(x_1, x_2) = 1,$
\item otherwise, $u(x_1, x_2) = -1$.
\end{itemize}
%%%%%%%%%%%%%%%%%%% стр. 66 %%%%%%%%%%%%%%%%%%%
Such a dependence $u(x)$ of optimal control on the current point $x$ is called optimal synthesis, it is the best possible form of solution to an optimal control problem.

\paragraph{Markov-Dubins car}
We have a time-optimal problem
\begin{align*}
&\dot{x} = \cos \theta, \quad q = (x ,y, \theta) \in \mathbb{R}^2_{x,y} \times S^1_{\theta} = M \\
&\dot{y} = \sin \theta, \quad |u| \le 1,\\
&\dot{\theta} = u,\\
&q(0) = q_0 = (0, 0, 0), \quad q(t_1) = q_1,\\
&t_1 \to \min.
\end{align*}
First we compute the orbit of the family $\mathcal{F} = \{ f(q, u)  \mid  |u| \le 1 \}$, where $f(q, u) = \cos \theta \frac{\partial}{\partial x} + \sin \theta \frac{\partial}{\partial y} + u \frac{\partial}{\partial \theta}$. 

We have $f(q, 0) = \cos \theta \frac{\partial}{\partial x} + \sin \theta \frac{\partial}{\partial y} \in \mathcal{F}$, $f(q, 1) - f(q, 0) = \frac{\partial}{\partial \theta} \in \Span(\mathcal{F})$, thus $[\cos \theta \frac{\partial}{\partial x} + \sin \theta \frac{\partial}{\partial y}, \frac{\partial}{\partial \theta}] = \sin \theta \frac{\partial}{\partial x} - \cos \theta \frac{\partial}{\partial y} \in \Lie_q (\mathcal{F})$.
%%%%%%%%%%%%%%%%%%% стр. 67 %%%%%%%%%%%%%%%%%%%
So $\Lie_q (\mathcal{F}) \supset \Span (\frac{\partial}{\partial x}, \frac{\partial}{\partial y}, \frac{\partial}{\partial \theta}) = T_q M$, thus $O_q = M$ for any $q \in M$.

Now we evaluate the attainable set $A_q$. Introduce, along with the system $\mathcal{F}$, a smaller system
\begin{align*}
&\mathcal{F}_1 = \{ f_0 + f_1, f_0 - f_1 \}, \\
&f_0 = f(q, 0) = \cos \theta \frac{\partial}{\partial x} + \sin \theta \frac{\partial}{\partial y}, \\
&f_1 = \frac{\partial}{\partial \theta}.
\end{align*}
Since $\mathcal{F}_1 \subset \mathcal{F}$, then $A_q (\mathcal{F}_1) \subset A_q (\mathcal{F})$ for any $q \in M$.

Compute the trajectories of the vector fields $f_0 \pm f_1$:
\begin{align*}
&u = \pm 1, \\
&\theta = \theta_0 \pm t,\\
&x = x_0 \pm (\sin(\theta_0 \pm t) - \sin \theta_0),\\
&y = y_0 \pm (\cos \theta_0 - \cos (\theta_0 \pm t)).
\end{align*}
These trajectories are $2\pi$-periodic, thus $e^{-t(f_0 \pm f_1)} = e^{(2\pi n - t)(f_0 \pm f_1)}$, i.e., any point
%%%%%%%%%%%%%%%%%%% стр. 68 %%%%%%%%%%%%%%%%%%%
attainable via the fields $f_0 \pm f_1$ in a negative time is attainable in a positive time as well. Consequently, $A_q (\mathcal{F}_1) = O_q (\mathcal{F}_1)$. But $O_q (\mathcal{F}_1) = M$ via Rashevsky-Chow theorem. So we get the chain
$$ A_q (\mathcal{F}) \supset A_q (\mathcal{F}_1) = O_q (\mathcal{F}_1) = M,$$
whence $A_q(\mathcal{F}) = M$ for any $q \in M$.

All conditions of Filippov theorem are satisfied: $U$ is compact, $f(q, u)$ are convex, the bound $|f(q, u)| \le 2$ implies apriori bound of the attainable set. Thus optimal control exists.

We apply PMP. The vector fields $f_0, f_1, f_2 = [f_0, f_1] = \sin \theta \frac{\partial}{\partial x} - \cos \theta \frac{\partial}{\partial y}$ form a frame in $T_q M$. Define the corresponding linear on fibers of $T^* M$ Hamiltonians:
$$ h_i (\lambda) = \langle \lambda, f_i \rangle, \quad i = 0, 1, 2.$$
The shortened Hamiltonian of PMP is 
$$ h_u (\lambda) = \langle \lambda, f_0 + u f_1 \rangle = h_0 + u h_1. $$
%%%%%%%%%%%%%%%%%%% стр. 69 %%%%%%%%%%%%%%%%%%%
The functions $h_0, h_1, h_2$ form a coordinate system on $T_q^* M$, and we write the Hamiltonian system of PMP in the parameterization $(h_0, h_1, h_2, q)$ of $T^* M$:
\begin{align*}
&\dot{h}_0 = \vec{h}_u h_0 = \{ h_0 + u h_1, h_0 \} = -u h_2,\\
&\dot{h}_1 = \{ h_0 + uh_1, h_1 \} = h_2,\\
&\dot{h}_2 = \{ h_0 + uh_1, h_2 \} = u h_0, \\
&\dot{q} = f_0 + u f_1.
\end{align*}
The maximality condition $h_u (\lambda) = h_0 + uh_1 \to \max\limits_{|u| \le 1}$ implies that if $h_1 (\lambda_t) \ne 0$, then $u(t) = \sgn h_1 (\lambda_t)$.

Consider the case where the control is not determined by PMP: $h_1 (\lambda_t) \equiv 0$ (this case is called singular). Then the Hamiltonian system gives $h_2 (\lambda_t) \equiv 0$, thus $h_0 (\lambda_t) \ne 0$, so $u(t) \equiv 0$. The corresponding extremal trajectory $(x(t), y(t))$ is a straight line.

If $u(t) = \pm 1$, then the extremal trajectory $(x(t), y(t))$ is an arc of a unit circle.
%%%%%%%%%%%%%%%%%%% стр. 70 %%%%%%%%%%%%%%%%%%%
One can show that optimal trajectories have one of the following two types:
\begin{enumerate}
\item arc of unit circle + line segment + arc of unit circle,
\item concatenation of arcs of unit circles with not more than 3 switchings; the angle of rotation between switchings is the same and belongs to $[\pi, 2\pi)$.
\end{enumerate}
If boundary conditions are far one from another, then optimal trajectory has type 1 and can explicitly be constructed as follows. Draw two unit circles that satisfy the initial condition and two unit circles that satisfy the terminal condition. Draw four common tangents to initial circles and terminal circles, with account of direction of motion along the circles determined by the boundary conditions. Among the four constructed extremal trajectories, find the shortest one. It is the optimal
%%%%%%%%%%%%%%%%%%% стр. 71 %%%%%%%%%%%%%%%%%%%
 trajectory.
 
 Optimal synthesis for the Dubins car is known, but it is rather complicated.
\paragraph{Euler elasticae}
We have the optimal control problem
\begin{align*}
&\dot{x} = \cos \theta, \quad q = (x, y, \theta) \in \mathbb{R}^2_{x,y} \times S^1_{\theta} = M,\\
&\dot{y} = \sin \theta, \quad u \in \mathbb{R}, \\
&\dot{\theta} = u,\\
&q(0) = q_0 = (0, 0, 0), \quad q(t_1) = q_1,\\
&t_1 \textrm{ is fixed},\\
&J = \frac12 \int^{t_1}_0 u^2 dt \to \min.
\end{align*}
Choosing appropriate unit of length in the plane $\mathbb{R}^2_{x,y}$, we can assume that $t_1 = 1$.

The control system in this example is the same as in the previous one, thus $O_{q_0} = M$.

Geometrically it is obvious that
%%%%%%%%%%%%%%%%%%% стр. 72 %%%%%%%%%%%%%%%%%%%
$$ A_{q_0} (1) = \{ q \in M  \mid  x^2 + y^2 < 1 \textrm{ or } (x, y, \theta) = (1, 0, 0) \}.$$
We suppose in  the sequel that $q_1 \in A_{q_0} (1)$. The set of control parameters $U = \mathbb{R}$ is noncompact, thus Filippov theorem is not applicable. One can show (using general existence results of optimal control theory) that optimal control exists.

Denote the frame on $M$:
\begin{align*}
&f_1 = \cos \theta \frac{\partial}{\partial x}+ \sin \theta \frac{\partial}{\partial y}, \\
&f_2 = \frac{\partial}{\partial \theta},\\
&f_3 = \sin \theta \frac{\partial}{\partial x}+ \sin \theta \frac{\partial}{\partial y},
\end{align*}
and introduce linear on fibers Hamiltonians --- coordinates on $T^*_q M$:
$$ h_i (\lambda) = \langle \lambda, f_i \rangle, \quad i = 1, 2, 3. $$
Then the Hamiltonian of PMP reads
$$ h^{\nu}_u (\lambda) = \langle \lambda, f_1 + u f_2 \rangle + \frac{\nu}{2}u^2 = h_1 + u h_2 + \frac{\nu}{2} u^2. $$
The corresponding Hamiltonian system of PMP 
%%%%%%%%%%%%%%%%%%% стр. 73 %%%%%%%%%%%%%%%%%%%
reads:
\begin{align*}
&\dot{h}_1 = \{ h_1 + u h_2, h_1 \} = - u h_3,\\
&\dot{h}_2 = \{ h_1 + u h_2, h_2 \} = h_3,\\
&\dot{h}_3 = \{ h_1 + u h_2, h_3 \} = u h_1,\\
&\dot{q} = f_1 + u f_2.
\end{align*}
The maximality condition of PMP is 
$$ h_1 + u h_2 + \frac{\nu}{2} u^2 \to \max\limits_{u \in \mathbb{R}}.$$
Consider first the abnormal case $\nu = 0$. Then the maximality condition $h_1 + u h_2 \to \max\limits_{u \in \mathbb{R}}$ yields $h_2 \equiv 0$, whence from the Hamiltonian system $h_3 \equiv 0$. Then from the nontriviality condition of PMP $h_1 \ne 0$. The Hamiltonian system yields $h_1 \equiv \const$, thus $u \equiv 0$.

The abnormal extremal trajectory is $q(t) = e^{t f_1} (q_0)$, it is projected to the line $(x, y) = (t, 0)$. It is optimal since in this case $J = 0 = \min$. 

Now consider the normal case $\nu = -1$. The maximality condition $h_1 + u h_2 - \frac{u^2}{2} \to \max\limits_{u \in \mathbb{R}}$ implies $u = h_2$, then the Hamiltonian
%%%%%%%%%%%%%%%%%%% стр. 74 %%%%%%%%%%%%%%%%%%%
system of PMP reads as follows:
\begin{align*}
&\dot{h}_1 = -h_2 h_3,\\
&\dot{h}_2 = h_3,\\
&\dot{h}_3 = h_2 h_1,\\
&\dot{q} = f_1 + h_2 f_2.
\end{align*}
This system has an integral $h^2_1 + h^2_3 \equiv \const$. Introduce the polar coordinates
$$ h_1 = r \cos \alpha, \quad h_3 = r \sin \alpha, $$
then the subsystem of the Hamiltonian system for the vertical variables $h_i$ reads as follows:
\begin{equation*}
 \begin{cases}
\dot{\alpha} = h_2,\\
\dot{h}_2 = r \sin \alpha.
\end{cases}
\end{equation*}
Denote $\beta = \alpha + \pi$, then we get the equation of pendulum:
\begin{equation*}
 \begin{cases}
\dot{\beta} = h_2,\\
\dot{h}_2 = -r \sin \beta.
\end{cases}
\end{equation*}
This equation has an energy integral $E = \frac{h^2_2}{2} - r \cos \beta$.
%%%%%%%%%%%%%%%%%%% стр. 75 %%%%%%%%%%%%%%%%%%%
The ODEs for the horizontal variables are as follows:
\begin{align*}
&\dot{x} = \cos \theta,\\
&\dot{y} = \sin \theta,\\
&\dot{\theta} = h_2 = \dot{\beta}, \quad \textrm{ thus } \theta = \beta - \beta_0.
\end{align*}
The shape of Euler elasticae $(x(t), y(t))$ is determined by values of the energy integral $E \in [-r, + \infty)$ and the corresponding motion of the pendulum.

If $E = -r < 0$, then the pendulum stays at the stable equilibrium $(\beta, h_2) = (0, 0)$, and the elastic curve is a straight line.

If $E \in (-r, r), r > 0$, then the pendulum oscillates, and Euler elasticae have inflection points.

If $E = r > 0$, then the pendulum either stays at the unstable equilibrium $(\beta, h_2) = (\pi, 0)$, or tends to it for an infinite time; correspondingly Euler elasticae are either straight line or a critical curve
%%%%%%%%%%%%%%%%%%% стр. 76 %%%%%%%%%%%%%%%%%%%
without inflection points and with one loop.

If $E > r > 0$, then the pendulum rotates in one or another direction, and elastic curves have no inflection points.

Finally, if $r = 0$, then the pendulum either rotates uniformly or stays fixed (in the absence of gravity); the elastic curves are respectively either circles or the straight line.

Although this problem was first considered in detail by Euler in 1742, optimal synthesis is still unknown.
\paragraph{Rolling of $S^2$ on $\mathbb{R}^2$ without slipping or twisting}
Prove that in this problem the sphere rolls optimally along Euler elasticae in the plane.
\section{Sub-Riemannian geometry}
\subsection{Sub-Riemannian structures and minimizers}
%%%%%%%%%%%%%%%%%%% стр. 77 %%%%%%%%%%%%%%%%%%%
A sub-Riemannian structure on a smooth manifold $M$ is a pair $(\Delta, g)$, where $\Delta$ is a distribution on $M$ and $g$ is an inner product (nondegenerate positive definite quadratic form) on $\Delta$.

A curve $q \in \Lip ([0, t_1], M)$ is called horizontal (admissible) if
$$ \dot{q}(t) \in \Delta_{q(t)} \textrm{ for almost all } t \in [0, t_1]. $$
The length of a horizontal curve is 
$$ l (q (\cdot)) = \int^{t_1}_0 \sqrt{g(\dot{q}, \dot{q})} dt.$$
Sub-Riemannian (Carno-Carath\'eodory) distance between points $q_0, q_1 \in M$ is 
$$d (q_0, q_1) = \inf \{ l (q(\cdot))  \mid  q(\cdot)  \textrm{ horizontal}, ~ q(0) = q_0, ~ q(t_1) = q_1\}.
$$
A horizontal curve $q(\cdot)$ is called a length minimizer if 
$$ l (q(\cdot)) = d (q(0), q(t_1)). $$
Thus length minimizers are solutions to an optimal control problem:
%%%%%%%%%%%%%%%%%%% стр. 78 %%%%%%%%%%%%%%%%%%%
\begin{align*}
&\dot{q} (t) \in \Delta_{q(t)},\\
&q(0) = q_0, \quad q(t_1) = q_1,\\
&l(q(\cdot)) \to \min.
\end{align*}
Suppose that a sub-Riemannian structure $(\Delta, g)$ has a global orthonormal frame $f_1, \dots, f_k \in \Vec(M)$:
$$ \Delta_q = \Span (f_1 (q), \dots, f_k (q)), \quad q \in M, \quad g(f_i, f_j) = \delta_{ij}, \quad i, j = 1, \dots, k. $$
Then the optimal control problem for sub-Riemannian minimizers takes the standard form:
\begin{align}
&\dot{q} = \sum^k_{i=1} u_i f_i (q), \quad q \in M, \quad u = (u_1, \dots, u_k) \in \mathbb{R}^k, \label{qdotui}\\
&q (0) = q_0, \quad q(t_1) = q_1 \nonumber, \\
&l = \int^{t_1}_0 \left(\sum^k_{i=1} u^2_i\right)^{1/2} dt \to \min \label{lmin}.
\end{align}
By Cauchy-Schwarz inequality, the length minimization problem \eqref{lmin} is equivalent to energy minimization problem
$$ J = \frac12 \int^{t_1}_0 \sum^k_{i=1} u^2_i dt \to \min. $$
%%%%%%%%%%%%%%%%%%% стр. 79 %%%%%%%%%%%%%%%%%%%
The energy functional $J$ is more convenient than the length functional $l$ since $J$ is smooth and its minimizers have constant velocity $\sum^k_{i=1} u^2_i \equiv \const$, while $l$ is not smooth when $\sum^k_{i=1} u^2_i = 0$, and its minimizers have arbitrary parameterization.

In the following several examples we present control systems \eqref{qdotui} for the corresponding sub-Riemannian structures.
\paragraph{Heisenberg group}
\begin{align*}
\begin{pmatrix}
\dot{x}\\
\dot{y}\\
\dot{z}\\
\end{pmatrix}
= u_1
\begin{pmatrix}
1 \\
0 \\
- \frac{y}{2}
\end{pmatrix}
+ u_2
\begin{pmatrix}
0 \\
1 \\
\frac{x}{2}
\end{pmatrix}
, \quad q \in \mathbb{R}^3_{x, y, z}, \quad u \in \mathbb{R}^2.
\end{align*}
\paragraph{Group of Euclidean motions of the plane}
\begin{align*}
\begin{pmatrix}
\dot{x}\\
\dot{y}\\
\dot{\theta}\\
\end{pmatrix}
= u_1
\begin{pmatrix}
\cos \theta \\
\sin \theta \\
0
\end{pmatrix}
+ u_2
\begin{pmatrix}
0 \\
0 \\
1
\end{pmatrix}
, \quad q \in \mathbb{R}^2 \times S^1 \cong \SE(2), \quad u \in \mathbb{R}^2.
\end{align*}
%%%%%%%%%%%%%%%%%%% стр. 80 %%%%%%%%%%%%%%%%%%%
\paragraph{Engel group}
\begin{align*}
\begin{pmatrix}
\dot{x}\\
\dot{y}\\
\dot{z}\\
\dot{v}\\
\end{pmatrix}
= u_1
\begin{pmatrix}
1 \\
0 \\
-\frac{y}{2}\\
-\frac{x^2 + y^2}{2}
\end{pmatrix}
+ u_2
\begin{pmatrix}
0 \\
1 \\
\frac{x}{2}\\
0
\end{pmatrix}
, \quad q \in \mathbb{R}^4_{x,y,z,v}, \quad u \in \mathbb{R}^2.
\end{align*}

\subsection{Lie algebra rank condition for SR problems}
The system $\mathcal{F} = \left\{ \sum^k_{i=1} u_i f_i  \mid  u_i \in \mathbb{R} \right\}$ is symmetric, thus $A_q = O_q$ for any $q \in M$. Assume that $M$ and $\mathcal{F}$ are real-analytic, and $M$ is connected. Then the system $\mathcal{F}$ is controllable if it has full rank:
$$ \Lie_q (\mathcal{F}) = \Lie_q (f_1, \dots, f_k) = T_q M, \quad q \in M. $$

\subsection{Filippov theorem for SR problems}
We can equivalently rewrite the optimal control problem for SR minimizers as the following time-optimal problem:
\begin{align*}
&\dot{q} = \sum^k_{i=1} u_i f_i (q), \quad \sum^k_{i=1} u^2_i \le 1, \quad q \in M,\\
&q(0) = q_0, \quad q(t_1) = q_1,\\
&t_1 \to \min.
\end{align*}
The set of control parameters $U = \{ u \in \mathbb{R}^k  \mid \sum^k_{i=1} u^2_i \le 1 \}$ is compact, and the sets of admissible velocities $\left\{ \sum^k_{i=1} u_i f_i (q)  \mid  u \in U\right\} \subset T_q M$ are convex. If we prove an apriori estimate for
%%%%%%%%%%%%%%%%%%% стр. 81 %%%%%%%%%%%%%%%%%%%
the attainable sets $A_{q_0} (\le t_1)$, then Filippov theorem guarantees existence of length minimizers.

\subsection{Pontryagin Maximum Principle for SR problems}
Introduce linear on fibers of $T^* M$ Hamiltonians $h_i (\lambda) = \langle \lambda, f_i \rangle, \quad i = 1, \dots, k$. Then the Hamiltonian of PMP for SR problem takes the form
$$ h^{\nu}_u (\lambda) = \sum^k_{i=1} u_i h_i (\lambda) + \frac{\nu}{2} \sum^k_{i=1} u^2_i. $$
%%%%%%%%%%%%%%%%%%% стр. 82 %%%%%%%%%%%%%%%%%%%
Normal case: $\nu = -1$. The maximality condition $ \sum^k_{i=1} u_i h_i - \frac12  \sum^k_{i=1} u^2_i \to \max\limits_{u_i \in \mathbb{R}}$ yields $u_i = h_i$, then the Hamiltonian takes the form
$$ h^{-1}_u (\lambda) = \frac12  \sum^k_{i=1} h^2_i (\lambda) = H (\lambda).$$
The function $H(\lambda)$ is called the normal maximized Hamiltonian. Since it is smooth, in the normal case extremals satisfy the Hamiltonian system $\dot{\lambda} = \vec{H} (\lambda)$.

Abnormal case: $\nu = 0$. The maximality condition $\sum^k_{i=1} u_i h_i \to \max\limits_{u_i \in \mathbb{R}}$ implies that $h_i (\lambda_t) \equiv 0, \quad i = 1, \dots, k$. Thus abnormal extremals satisfy the conditions:
\begin{align*}
&\dot{\lambda}_t = \sum^k_{i=1} u_i (t) \vec{h}_i (\lambda_t),\\
&h_1 (\lambda_t) = \dots = h_k (\lambda_t) \equiv 0.
\end{align*}
\begin{remark}
Normal length minimizers are projections
%%%%%%%%%%%%%%%%%%% стр. 83 %%%%%%%%%%%%%%%%%%%
of solutions to the Hamiltonian system $\dot{\lambda} = \vec{H}(\lambda)$, thus they are smooth. An important open question of sub-Riemannian geometry is whether abnormal length minimizers are smooth.
\end{remark}
\subsection{Optimality of SR extremal trajectories}
In this subsection we consider normal extremal trajectories $q(t) = \pi (\lambda_t), \dot{\lambda}_t = \vec{H}(\lambda_t )$.

A horizontal curve $q(t)$ is called a SR geodesic if $g (\dot{q}, \dot{q}) \equiv \const$ and short arcs of $q(t)$ are optimal.
\begin{theorem}[Legendre]
Normal extremal trajectories are SR geodesics.
\end{theorem}
\paragraph{Example: Geodesics on $S^2$}
Consider the standard sphere $S^2 \subset \mathbb{R}^3$
%%%%%%%%%%%%%%%%%%% стр. 84 %%%%%%%%%%%%%%%%%%%
with Riemannian metric induced by the Euclidean metric of $\mathbb{R}^3$. Geodesics starting from the North pole $N \in S^2$ are great circles passing through $N$. Such geodesics are optimal up to the South pole $S \in S^2$. Variation of geodesics passing through $N$ yields the fixed point $S$, thus $S$ is a conjugate point to $N$. On the other hand, $S$ is the intersection point of different  geodesics of the same length starting at $N$, thus $S$ is a Maxwell point. In this example, conjugate point coincides with Maxwell point due to the one-parameter group of symmetries (rotations of $S^2$ around the line $NS$). In order to separate these points, one should destroy the rotational symmetry as in the following example.
%%%%%%%%%%%%%%%%%%% стр. 85 %%%%%%%%%%%%%%%%%%%
\paragraph{Example: Geodesics on ellipsoid}
Consider a three-axes ellipsoid  with the Riemannian metric induced by the Euclidean metric of the ambient $\mathbb{R}^3$. Consider the family of geodesics on the ellipsoid starting from a vertex $N$, and let us look at this family from the opposite vertex $S$. The family of geodesics has an envelope --- astroid centered at $S$. Each point of the astroid is a conjugate point; at such points the geodesics lose their local optimality. On the other hand, there is a segment joining a pair of opposite vertices of the astroid, where pairs of geodesics of the same length meet one another. This segment (except its vertices) consists of Maxwell points. At such points geodesics on the ellipsoid
%%%%%%%%%%%%%%%%%%% стр. 86 %%%%%%%%%%%%%%%%%%%
lose their global optimality.

We will clarify below the notions and facts that appeared in this example.

\bigskip

Consider the normal Hamiltonian system of PMP $\dot{\lambda}_t = \vec{H}(\lambda_t)$. The Hamiltonian $H$ is an integral of this system. We can assume that $H(\lambda_t) \equiv \frac12$, this corresponds to arclength parameterization of normal geodesics. Denote the cylinder $C = T_{q_0}^* M \cap \{ H = \frac12 \}$ and define the exponential mapping
\begin{align*}
&\Exp \colon C \times \mathbb{R}_+ \to M,\\
&\Exp (\lambda_0, t) = \pi \circ e^{t \vec{H}} (\lambda_0) = q(t).
\end{align*}
A point $q_1 = \Exp(\lambda_0, t_1)$ is a conjugate point along the geodesic $q(t) = \Exp(\lambda_0, t)$ if it is a critical value of $\Exp$: $\Exp_{* (\lambda_0, t_1)}$ is degenerate.
\begin{theorem}[Jacobi]
Let a normal geodesic $q(t)$ be a projection of 
%%%%%%%%%%%%%%%%%%% стр. 87 %%%%%%%%%%%%%%%%%%%
a unique, up to a scalar multiple, extremal. Then $q(t)$ loses its local optimality at the first conjugate point.
\end{theorem}
A point $q_1 = \Exp(\lambda_0, t_1)$ is conjugate iff the Jacobian $J (t_1) = \det (\frac{\partial \Exp}{\partial (\lambda_0, t)}) |_{t = t_1} = 0$. 

A point $q_1 = q(t_1)$ is a Maxwell point along a geodesic $q(t) = \Exp(\lambda_0, t)$ iff there exists another geodesic $\tilde{q}(t) = \Exp(\tilde{\lambda}_0, t) \not\equiv q(t)$ such that $q_1 = \tilde{q} (t_1)$.
\begin{lemma}
If $H$ is analytic, then a normal geodesic cannot be optimal after a Maxwell point.
\end{lemma}
\begin{proof}
Let $q_1 = q(t_1)$ be a Maxwell point along a geodesic $q(t) = \Exp(\lambda_0, t)$, and let $\tilde{q}(t) = \Exp(\tilde{\lambda}_0, t) \not\equiv q(t)$ be another
%%%%%%%%%%%%%%%%%%% стр. 88 %%%%%%%%%%%%%%%%%%%
geodesic with $\tilde{q} (t_1) = q_1$. If $q(t), ~ t \in [0, t_1 + \varepsilon],~\varepsilon > 0$, is optimal, then the following curve is optimal as well:
\begin{align*}
\bar{q}(t) = 
\begin{cases}
\tilde{q} (t), \quad t \in [0, t_1],\\
q(t), \quad t \in [t_1, t_1 + \varepsilon].
\end{cases}
\end{align*}
The geodesics $q(t)$ and $\bar{q}(t)$ coincide at the segment $t \in [t_1, t_1 + \varepsilon]$. Since they are analytic, they should coincide at the whole domain $t \in [0, t_1 + \varepsilon]$. Thus $q(t) \equiv \tilde{q} (t),~ t \in [0, t_1]$, a contradiction.
\end{proof}
\begin{theorem}
Let $q(t)$ be a normal geodesic that is a projection of a unique, up to a scalar multiple, extremal. Then $q(t)$ loses its global optimality either at the first Maxwell point or at the first conjugate point (at the first point of these two points).
\end{theorem}
%%%%%%%%%%%%%%%%%%% стр. 89 %%%%%%%%%%%%%%%%%%%
\subsection{Symmetry method for construction of optimal synthesis}
We describe a general method for construction of optimal synthesis for sub-Riemannian problems with a big group of symmetries (e.g. for left-invariant SR problems on Lie groups). We assume that for any $q_1 \in M$ there exists a length minimizer $q(t)$ that connects $q_0$ and $q_1$.

Suppose for simplicity that there are no abnormal geodesics. Thus all geodesics are parameterized by the normal exponential mapping
$$ \Exp \colon N \to M, \quad N = C \times \mathbb{R}_+. $$
If this mapping is bijective, then any point $q_1 \in M$ is connected with $q_0$ by a unique geodesic $q(t)$, and by virtue of existence of length minimizers this geodesic is optimal.

But typically the exponential mapping is not bijective due to Maxwell points.
%%%%%%%%%%%%%%%%%%% стр. 90 %%%%%%%%%%%%%%%%%%%
Denote by $t_{\max} (\lambda_0) \in (0, +\infty]$ the first Maxwell time along the geodesic $\Exp(\lambda_0, t)$. Consider the Maxwell set in the image of the exponential mapping
$$ \Max = \{ \Exp(\lambda_0, t_{\max}(\lambda_0))  \mid \lambda_0 \in C \},$$
and introduce the restricted exponential mapping
\begin{align*}
&\Exp \colon \widetilde{N} \to \widetilde{M},\\
&\widetilde{N} = \{ (\lambda_0, t) \in N  \mid  t < t_{\max}(\lambda_0) \},\\
&\widetilde{M} = M \backslash \cl(\Max).
\end{align*}
This mapping may well be bijective, and if this is the case, then any point $q_1 \in \widetilde{M}$ is connected with $q_0$ by a unique candidate optimal geodesic; by virtue of existence, this geodesic is optimal.

The bijective property of the restricted exponential mapping can often be proved via the 
%%%%%%%%%%%%%%%%%%% стр. 91 %%%%%%%%%%%%%%%%%%%
following theorem.
\begin{theorem}[Hadamard]
Let $F \colon X \to Y$ be a smooth mapping between smooth manifolds such that the following properties fold:
\begin{itemize}
\item $\dim X = \dim Y$,
\item $X, Y$ are connected and $Y$ is simply connected,
\item $F$ is nondegenerate,
\item $F$ is proper ($F^{-1}(K)$ is compact for compact $K \subset Y)$.
\end{itemize}
Then $F$ is a diffeomorphism.
\end{theorem}
Usually it is hard to describe all Maxwell points (and respectively to describe the first of them), but one can do this for a group of symmetries $G$ of the exponential mapping. A pair of mappings $\varepsilon \colon N \to N$, $\sigma \colon M \to M$ is called
%%%%%%%%%%%%%%%%%%% стр. 92 %%%%%%%%%%%%%%%%%%%
a symmetry of the exponential mapping if $\sigma \circ \Exp = \Exp \circ \varepsilon$. Suppose that there is a group $G$ of symmetries of the exponential mapping consisting of reflections $\varepsilon \colon N \to N$ and $\sigma \colon M \to M$. If a point $q_1 = \Exp (\lambda_0, t)$ is a fixed point for some $\sigma \in G$ such that $(\lambda_0, t)$ is not a fixed point for the corresponding $\varepsilon \in G$, then $q_1$ is a Maxwell point. In such a way one can describe the Maxwell points corresponding to the group of symmetries $G$, and consequently describe the first Maxwell time corresponding to the group $G$, $t^G_{\max} \colon C \to (0, +\infty]$. Then one can apply the above procedure with the restricted exponential mapping, replacing $t_{\max}(\lambda_0)$ by $t^G_{\max}(\lambda_0)$ . If the group $G$ is big enough,
%%%%%%%%%%%%%%%%%%% стр. 93 %%%%%%%%%%%%%%%%%%%
one can often prove that the restricted exponential mapping is bijective, and thus to construct optimal synthesis.
\subsection{Sub-Riemannian problem on the Heisenberg group}
The problem is stated as follows:
\begin{align*}
&\dot{q} = u_1 f_1 (q) + u_2 f_2 (q), \quad q \in M = \mathbb{R}^3_{x,y,z}, \quad u \in \mathbb{R}^2,\\
&q(0) = q_0 = (0, 0, 0),~ q(t_1) = q_1,\\
&J = \frac12 \int^{t_1}_0 (u^2_1 + u^2_2) dt \to \min,\\
&f_1 = \frac{\partial}{\partial x} - \frac{y}{2} \frac{\partial}{\partial z}, \quad f_2 = \frac{\partial}{\partial y} + \frac{x}{2} \frac{\partial}{\partial z}.
\end{align*}
We have $[f_1, f_2] = f_3 = \frac{\partial}{\partial z}$. The system has full rank, thus it is completely controllable.

The right-hand side satisfies the bound
$$ | u_1 f_1 (q) + u_2 f_2 (q)| \le C(1 + |q|), \quad q \in M, \quad u^2_1 + u^2_2 \le 1. $$
Thus Filippov theorem gives existence
%%%%%%%%%%%%%%%%%%% стр. 94 %%%%%%%%%%%%%%%%%%%
of optimal controls.

Introduce linear on fibers of $T^* M$ Hamiltonians:
$$ h_i (\lambda) = \langle \lambda, f_i \rangle, \quad i = 1, 2, 3, \quad \lambda \in T^* M. $$
Abnormal case: abnormal extremals satisfy the Hamiltonian system $\dot{\lambda} = u_1 \vec{h}_1 (\lambda) + u_2 \vec{h}_2 (\lambda)$, in coordinates:
\begin{align*}
&\dot{h}_1 = -u_2 h_3,\\
&\dot{h}_2 = u_1 h_3,\\
&\dot{h}_3 = 0,\\
&\dot{q} = u_1 f_1 + u_2 f_2,
\end{align*}
plus the identities 
$$ h_1 (\lambda_t) = h_2 (\lambda_t) \equiv 0. $$
Thus $h_3 (\lambda_t) \ne 0$, and the first two equations of the Hamiltonian system yield $u_1 (t) = u_2 (t) \equiv 0$. Thus abnormal trajectories are constant.

Normal case: normal extremals satisfy the Hamiltonian system $\dot{\lambda} = \vec{H}(\lambda)$, in coordinates:
%%%%%%%%%%%%%%%%%%% стр. 95 %%%%%%%%%%%%%%%%%%%
\begin{align*}
&\dot{h}_1 = -h_2 h_3,\\
&\dot{h}_2 = h_1 h_3,\\
&\dot{h}_3 = 0,\\
&\dot{q} = h_1 f_1 + h_2 f_2.
\end{align*}
On the level surface $H = \frac12 (h^2_1 + h^2_2) \equiv \frac12$, we introduce the coordinate $\theta$:
$$ h_1 = \cos \theta, \quad h_2 = \sin \theta. $$
Then the normal Hamiltonian system takes the form:
\begin{align*}
&\dot{\theta} = h_3,\\
&\dot{h}_3 = 0,\\
&\dot{x} = \cos \theta,\\
&\dot{y} = \sin \theta,\\
&\dot{z} = - \frac{y}{2} \cos \theta + \frac{x}{2} \sin \theta,\\
&(x, y, z)(0) = (0, 0, 0).
\end{align*}
\begin{enumerate}
\item If $h_3 = 0$, then
\begin{align*}
&\theta \equiv \theta_0,\\
&x = t\cos \theta_0,\\
&y = t\sin \theta_0,\\
&z = 0.
\end{align*}
%%%%%%%%%%%%%%%%%%% стр. 96 %%%%%%%%%%%%%%%%%%%
\item If $h_3 \ne 0$, then
\begin{align*}
&\theta = \theta_0 + h_3 t,\\
&x = (\sin (\theta_0 + h_3 t) - \sin \theta_0) / h_3,\\
&y = (\cos \theta_0 - \cos (\theta_0 + h_3 t)) / h_3,\\
&z = (h_3 t - \sin h_3 t) / h^2_3.
\end{align*}
\end{enumerate}
If  $h_3 = 0$, then the geodesic $q(t)$ is optimal for $t \in [0, + \infty)$ since its projection to the plane $(x, y)$ is a line, and the minimized functional is the Euclidean length in $(x, y)$. 

In the case $h_3 \ne 0$ we study first local optimality by evaluation of conjugate points:
$$ J(t) = \frac{\partial \Exp}{\partial (\lambda_0, t)} = \frac{\partial (x, y, z)}{\partial (\theta_0, h_3, t)}. $$
In the coordinates $p = \frac{h_3 t}{2}$, $\tau = \theta_0 + \frac{h_3 t}{2}$, we have:
\begin{align*}
&x = \frac{2}{h_3} \cos \tau \sin p, \\
&y = \frac{2}{h_3} \sin \tau \sin p, \\
%%%%%%%%%%%%%%%%%%% стр. 97 %%%%%%%%%%%%%%%%%%%
&z = \frac{2p - \sin 2p}{h^2_3}.
\end{align*}
Thus 
\begin{align*}
 &J(p) = \frac{\partial (x, y, z)}{\partial (\tau, p, h_3)} = \frac{8\sin p}{h^5_3} \cdot \varphi(p),\\
 &\varphi(p) = (2p - \sin 2p) \cos p - (1 - \cos 2p) \sin p.
 \end{align*}
 The function $\varphi(p)$ does not vanish for $p \in (0, \pi)$, thus the first root of $J(p)$ is $p^1_{\conj} = \pi$. Summing up, the first conjugate time in the case $h_3 \ne 0$ is 
 $$ t^1_{\conj} = \frac{2\pi}{|h_3|}. $$
 The problem has an obvious symmetry group --- rotations around the $z$-axis. The corresponding Maxwell times are $t = \frac{2\pi n}{h_3}$, and Maxwell points in the image of the exponential mapping are $x = y = 0,$ $z = \frac{2 \pi n}{h^2_3}$. The first Maxwell time corresponding to the group of rotations is $t^1_{\max} = \frac{2 \pi}{|h_3|} = t^1_{\conj}$.
 %%%%%%%%%%%%%%%%%%% стр. 98 %%%%%%%%%%%%%%%%%%%
 Consider the restricted exponential mapping 
 \begin{align*}
 &\Exp \colon \widetilde{N} \to \widetilde{M}, \\
 &\widetilde{N} = \{ (\lambda, t) \in N  \mid  \theta \in S^1, \quad h_3 > 0, \quad t \in (0, \frac{2 \pi}{h_3}) \},\\
 &\widetilde{M} = \{ q \in M  \mid  z > 0, \quad x^2 + y^2 > 0 \}.
 \end{align*}
 The mapping $\Exp|_{ \widetilde{N}}$ is nondegenerate and proper $((\theta, h_3, t) \to \partial  \widetilde{N} \Rightarrow q \to \partial  \widetilde{M})$. The manifolds $\widetilde{N}$, $\widetilde{M}$ are connected, but $\widetilde{M}$ is not simply connected. Thus Hadamard theorem cannot be applied immediately. In order to pass to simply connected manifold, we factorize the exponential mapping by the group of rotations. We get
 \begin{align*} 
&\widehat{N} = \widetilde{N} / S^1 = \{ (h_3, t) \in \mathbb{R}^2  \mid  h_3 > 0, \quad t \in (0, \frac{2\pi}{h_3}) \},\\
&\widehat{M} = \widetilde{M} / S^1= \{ (r, z) \in \mathbb{R}^2  \mid  z > 0, \quad r = \sqrt{x^2 + y^2} > 0 \},\\
&\widehat{\Exp} \colon \widehat{N} \to \widehat{M}, \quad \widehat{\Exp} (h_3, t) = (z, r),\\
 %%%%%%%%%%%%%%%%%%% стр. 99 %%%%%%%%%%%%%%%%%%%
& z = \frac{2p - \sin 2p}{h^2_3}, \quad r = \frac{2}{h_3} \sin p, \quad p = \frac{h_3 t}{2}.
 \end{align*}
 By Hadamard theorem, the mapping $\widehat{\Exp} \colon \widehat{N} \to \widehat{M}$ is a diffeomorphism, thus  $\Exp \colon \widetilde{N} \to \widetilde{M}$  is a diffeomorphism as well. 
 
 So for any $q_1 \in \widetilde{M}$ there exists a unique $(\lambda_0, t) \in \widetilde{N}$ such that $q_1 = \Exp (\lambda_0, t_1)$. Thus the geodesic $q(t) = \Exp(\lambda_0, t)$, $t \in [0, t_1]$, is optimal. Summing up, if $z_1 \ne 0$, $x^2_1 + y^2_1 \ne 0$, then there exists a unique minimizer connecting $q_0$ with $q_1 = (x_1, y_1, z_1)$, it is determined by parameters $\theta_0 \in S^1,$ $h_3 \ne 0,$ $t_1 \in (0, \frac{2\pi}{|h_3|})$.
 
 If $z_1 = 0$, $x^2_1 + y^2_1 \ne 0$, then there exists a unique minimizer determined by parameters $\theta_0 \in S^1,$ $h_3 = 0,$ $t_1 > 0.$
 
 Finally, if $z_1 \ne 0$, $x^2_1 + y^2_1 = 0$, then 
  %%%%%%%%%%%%%%%%%%% стр. 100 %%%%%%%%%%%%%%%%%%%
  there exists a one-parameter family of minimizers determined by parameters $\theta_0 \in S^1,$ $h_3 \ne 0,$ $t_1 = \frac{2\pi}{|h_3|}$.

  Let us describe the SR distance $d_0 (q) = d(q_0, q)$.
  
  If $z = 0$, then $d_0 (q) = \sqrt{x^2 + y^2}$.
  
    If $z \ne 0$, $x^2 + y^2 = 0$, then $d_0 (q) = \sqrt{2\pi |z|}$.
    
    If $z \ne 0$, $x^2 + y^2 \ne 0$, then the distance is determined by the conditions
\begin{align*} 
\begin{cases}
d_0 (q) = \frac{p}{\sin p} \sqrt{x^2 + y^2},\\
\frac{2p - \sin 2p}{4 \sin^2 p} = \frac{z}{x^2 + y^2}.
\end{cases}
 \end{align*}
Exercise: Show that $d_0 \in C (\mathbb{R}^3)$.

\end{document}